\documentclass[reqno,letterpaper]{amsart}
\usepackage{amsmath}
\usepackage{amssymb}
\usepackage{mcreek_custom_commands_v2}
\usepackage{hyperref}
\usepackage{cleveref}
\crefname{section}{section}{sections}
\crefname{subsection}{subsection}{subsections}
\newtheoremstyle{example}{}{}{}{}{\bfseries}{.}{.5em}{}
\newtheorem{theorem}{Theorem}[section]
\crefname{theorem}{Theorem}{Theorems}
\newtheorem{lemma}[theorem]{Lemma}
\crefname{lemma}{Lemma}{Lemmata}
\newtheorem{cor}[theorem]{Corollary}
\crefname{cor}{Corollary}{Corollaries}
\newtheorem{prop}[theorem]{Proposition}
\crefname{prop}{Proposition}{Propositions}
\newtheorem{remark}[theorem]{Remark}
\crefname{remark}{Remark}{Remarks}
\theoremstyle{definition}
\newtheorem{defn}[theorem]{Definition}
\crefname{defn}{Definition}{Definitions}
\theoremstyle{example}

\crefname{ex}{Example}{Examples}
\numberwithin{equation}{section}
\begin{document}
\title{Large-Data Global Well-Posedness for the $\(1 + 2\)$-Dimensional Equivariant Faddeev Model}
\author{Matthew Creek}
\begin{abstract}The Faddeev model is a classical field theory that models heavy elementary particles by knotted topological solitons. It is a generalization of the well-known classical nonlinear sigma model of Gell-Mann and Levy, and is also related closely to the celebrated Skyrme model. The global well-posedness of the quasilinear PDE arising from this model has been studied intensely in recent years, both in three and two spatial dimensions. In this paper we introduce a proof of large-data global well-posedness of the two-dimensional Faddeev model under the equivariant hypothesis.\end{abstract}
\maketitle
\section{Introduction}\label{sec:intro}
The Faddeev model is a classical field theory that models heavy elementary particles by knotted topological solitons. It was introduced by L. D. Faddeev in \cite{MR0495948,MR1553682}, and is intimately related to both the classical nonlinear $\sigma$ model and the Skyrme model (see, for example, \cite{MR2068924} and references therein). This theory concerns the formal critical points for the action functional $L_F = L_F\(\Omega\)$ defined by
\begin{equation}\label{L_F_def}L_F \defined \int_{\R^{1 + 3}}\[\frac{1}{2}\partial_\mu\Omega \cdot \partial^\mu\Omega + \frac{1}{4}\alpha^2\(\partial_\mu\Omega \times \partial_\nu\Omega\) \cdot \(\partial^\mu\Omega \times \partial^\nu\Omega\)\] \, dm\comma\end{equation}
where $\Omega : \(\R^{1 + 3}, m\) \to \(S^2, g\)$, $m$ is the Minkowski metric on $\R^{1 + 3}$, defined by
\begin{equation}\label{minkowski_metric_def}m \defined -dt^2 + \sumc{\mu}{1}{3}{dx_\mu^2}\comma\end{equation}
$g$ is the metric that $S^2$ inherits as a subspace of $\R^3$, $\alpha$ is a coupling constant having the dimension of length, and $\times$ is the cross product in $\R^3$. One should note here that, replacing $\R^{1 + 3}$ with $\R^{1 + n}$ and updating the metric $m$ accordingly, $L_F$ makes sense even in the more general setting of $\R^{1 + n}$. This allows us to define the action functionals for the $\(1 + n\)$-dimensional Faddeev models:
\begin{equation}\label{L_F_n_def}L_{F, n} \defined \int_{\R^{1 + n}}\[\frac{1}{2}\partial_\mu\Omega \cdot \partial^\mu\Omega + \frac{1}{4}\alpha^2\(\partial_\mu\Omega \times \partial_\nu\Omega\) \cdot \(\partial^\mu\Omega \times \partial^\nu\Omega\)\] \, dm\period\end{equation}
The static features of the Faddeev model have been studied extensively (see, for example, \cite{MR1660167,MR1809362}, and references therein). More recently, in \cite{MR2754038} and \cite{arXiv:1307.4721}, the dynamical features of the $\(1 + 2\)$-dimensional Faddeev model have also been investigated. Both of these results concern small-data global regularity for the variational equation corresponding to the $n = 2$ case of \eqref{L_F_n_def}. The former addresses smooth data and the latter addresses rough data in critical Besov spaces. In this paper we study the large-data problem for the $\(1 + 2\)$-dimensional Faddeev model under the equivariant hypothesis
\begin{equation}\label{2d_equivariant_hyp}\ob\begin{array}{l}\Omega = \Omega\(t, r, \omega\) = \(u\(t, r\), \omega\) = \(u, \omega\) \in \[0, \pi\] \times S^1\comma \\ \(u\(t, 0\), \fulllimit{r}{\infty}{u\(t, r\)}\) = \(0, \pi\)\comma \\ m = -dt^2 + dr^2 + r^2 \, d\omega^2\comma \quad \mbox{and} \quad g = du^2 + \sin^2u \, d\omega^2\period\end{array}\right.\end{equation}
Our ultimate goal is to extend the results in \cite{MR2754038} and \cite{arXiv:1307.4721} to the large-data regime. In a connected work, D. Li proved a large-data global regularity result for the classical equivariant Skyrme model. Our approach is to apply the techniques used by Li in \cite{arXiv:1208.4977} to prove a similar result for our problem. Under the equivariant hypothesis \eqref{2d_equivariant_hyp}, the Euler-Lagrange system of equations associated with the $n = 2$ case of the action functional \eqref{L_F_n_def} reduces to the following quasilinear equation for the azimuthal angular variable $u$:
\begin{equation}\label{u_intro_pde}\(1 + \alpha^2r^{-2}\sin^2u\)\(u_{tt} - u_{rr} - r^{-1}u_r\) = -\(2r^{-1}\alpha^2u_r\sin^2u + r^{-2}\[\alpha^2\(u_t^2 - u_r^2\) + 1\]\sin u\cos u\)\comma\end{equation}
with the boundary conditions also specified in \eqref{2d_equivariant_hyp}. Our main result in this paper is the following theorem.
\begin{theorem}\label{main_theorem}Let
\begin{equation}\label{main_theorem_regularity_hypothesis}s \ge 4 \quad \mbox{and} \quad \(u_0, u_1\) \in H^s_\mathrm{rad}\(\R^2\) \times H^{s - 1}_\mathrm{rad}\(\R^2\)\period\end{equation}
Then there is a unique
\begin{equation}\label{main_theorem_u_defn}u \in C_t\(\[0, \infty\), H_{x, \mathrm{rad}}^s\(\R^2\)\) \intersection C_t^1\(\[0, \infty\), H_{x, \mathrm{rad}}^{s - 1}\(\R^2\)\)\end{equation}
which solves \eqref{u_intro_pde} on $\[0, \infty\) \times \R^2$ such that $u$ satisfies
\begin{equation}\label{main_theorem_initial_and_boundary_condition_hypothesis}\(u\(0\), u_t\(0\)\) = \(u_0, u_1\) \quad \mbox{and} \quad \(u\(t, 0\), \fulllimit{r}{\infty}{u\(t, r\)}\) = \(\pi, 0\)\period\end{equation}
\end{theorem}
Our proof of this theorem proceeds in several steps. In \Cref{sec:lwp_u} we make the classical substitution $u = rv + \phi$ and recast \eqref{u_intro_pde} as a semilinear wave equation for a new unknown $v$ on $\R^{1 + 4}$ instead of $\R^{1 + 2}$. From there we use a slight modification of a classical local-existence theorem of L. H\"ormander that also provides a continuation criterion. In order to work with this continuation criterion, we wish to use a bootstrap technique with Strichartz estimates. However, the equation for $v$ is not amenable to this technique. Therefore, following the work of Li in \cite{arXiv:1208.4977}, we make, in \Cref{sec:der_Phi}, another substitution and replace $v$ with yet another unknown $\Phi$. This $\Phi$ is more amenable to the bootstrap technique, and is also so closely related to $v$ that the estimates we obtain for $\Phi$ generally are applicable to $v$ as well. In \Cref{sec:Y_1} we work directly with $\Phi$ and its wave equation to place it in an $H^1$-type Sobolev space. At this point we introduce the bootstrap technique with the Strichartz estimates and eventually, in \Cref{sec:Y_2} - \Cref{sec:Y_4}, upgrade $\Phi$ to an $H^4$-type Sobolev space. This allows us to place $\Phi$ and all of its first-order derivatives in $L^\infty$, and similarly for $v$. In \Cref{sec:final_steps} we show this is enough to satisfy the continuation criterion for $v$, which proves that the semilinear wave equation for $v$ is globally well-posed. Finally, we appeal to the relationship between $u$ and $v$ to conclude that \eqref{u_intro_pde} is therefore also globally well-posed, proving \Cref{main_theorem}.
\section{Local Well-Posedness for \eqref{u_intro_pde}}\label{sec:lwp_u}
In this section we first try to fit \eqref{u_intro_pde} into the framework provided by Li. To do this, we replace $u$ with $\pi - u$. We then introduce the new unknown $v$ and lift to dimension $\R^{1 + 4}$. After this we apply the (slight modification of the) classical theorem of H\"ormander to get local well-posedness with continuation criterion for $v$. We also introduce a ``toolbox'' of results that we use throughout the remainder of this paper.
\subsection{Derivation of Equation}\label{subsec:der_u_pde}
As indicated above, we replace $u$ with $\pi - u$. We then define
\begin{equation}\label{A_1_def}A_1 = A_1\(t, r\) \defined 1 + \alpha^2r^{-2}\sin^2u\comma\end{equation}
the definition in the limiting case $r = 0$ being the natural one. Next we define the general nonlinearity $\NNN$ by
\begin{equation}\label{NNN_def}\NNN\(u\) \defined -2r^{-1}\(1 - A_1^{-1}\)u_r - r^{-2}A_1^{-1}\[\alpha^2\(u_t^2 - u_r^2\) + 1\]\sin u\cos u\period\end{equation}
With this notation \eqref{u_intro_pde} can be recast as the boundary-value problem
\begin{equation}\label{u_pde}\square u = \NNN\(u\)\comma \quad \(u\(t, 0\), \fulllimit{r}{\infty}{u\(t, r\)}\) = \(\pi, 0\)\period\end{equation}
This is similar to the form used in \cite{arXiv:1208.4977} with parameter $N_1 = 1$. We want to appeal to standard semilinear PDE theory in order to get a local well-posedness result with a continuation criterion. However, we cannot do this directly with \eqref{u_pde} since it, being equivalent to \eqref{u_intro_pde}, is quasilinear. In order to resolve this, we apply a common approach: a change of variable and a lift of two dimensions. We let $\phi : \R^{1 + 4} \to \R$ be a smooth, radial, time-independent, monotone-decreasing cutoff function such that $\phi\(r\) = \pi$ for $r \le 1$ and $\phi\(r\) = 0$ for $r \ge 2$. Then we define the radial map $v : \R^{1 + 4} \to \R$ by
\begin{equation}\label{v_def}v \defined r^{-1}\(u - \phi\)\period\end{equation}
After also introducing $\phi_{< 1} : \R^{1 + 4} \to \R$, a smooth, radial, time-independent, monotone-decreasing cutoff function such that $\phi_{< 1}\(r\) = 1$ for $r \le \frac{1}{2}$ and $\phi_{< 1}\(r\) = 0$ for $r \ge 1$ and defining $\phi_{> 1} \defined 1 - \phi_{< 1}$, we combine \eqref{A_1_def} - \eqref{v_def} to obtain the following wave equation for $v$:
\begin{equation}\label{box_v_with_cutoff_without_u}\square v = r^{-2}v + r^{-1}\laplacian\phi + \phi_{< 1}r^{-1}\NNN\(rv + \pi\) + \phi_{> 1}r^{-1}\NNN\(rv + \phi\)\period\end{equation}
Immediately we see potential problems with the first and third terms on the RHS of \eqref{box_v_with_cutoff_without_u}. We will work initially on the third term. By working carefully with this term, we will also eliminate the $r^{-2}$ factor in the first term. We define the analytic functions $\tilde{F}_i : \R \to \R$, $0 \le i \le 4$ by
\begin{align}
\label{F_tilde_0_def} \tilde{F}_0\(x\) & \defined \alpha^2x^{-2}\sin^2x\comma \\
\label{F_tilde_1_def} \tilde{F}_1\(x\) & \defined x^{-2}\(1 - x^{-1}\sin x\cos x\)\comma \\
\label{F_tilde_2_def} \tilde{F}_2\(x\) & \defined \alpha^2x^{-3}\sin x\(\cos x - x^{-1}\sin x\)\comma \\
\label{F_tilde_3_def} \tilde{F}_3\(x\) & \defined -\alpha^2x^{-1}\sin x\cos x\comma \\
\label{F_tilde_4_def} \tilde{F}_4\(x\) & \defined 2\alpha^2x^{-3}\sin x\(\cos x - x^{-1}\sin x\)\comma
\end{align}
(the $\tilde{F}_i$ being defined at $x = 0$ by their limits) and the operators $\FFF_i$, $1 \le i \le 4$ by
\begin{equation}\label{FFF_i_def}\FFF_1\(v\) \defined v^3\comma \quad \FFF_2\(v\) = v^5\comma \quad \FFF_3\(v\) = v\(v_t^2 - v_r^2\)\comma \quad \mbox{and} \quad \FFF_4\(v\) = rv^4v_r\period\end{equation}
A very careful calculation then yields
\begin{equation}\label{box_v_third_term_with_F_tilde_simplified}\phi_{< 1}r^{-1}\NNN\(rv + \pi\) = \frac{\phi_{< 1}}{1 + \(\tilde{F}_0 \circ rv\)v^2}\sumc{j}{1}{4}{\(\tilde{F}_j \circ rv\)\FFF_j\(v\)} - \phi_{< 1}r^{-2}v\period\end{equation}
\begin{remark}\label{F_j_analyticity}It is clear from their definitions that all of the $\tilde{F}_j$ are even and analytic and, moreover, that $\fulllimit{\abs{x}}{\infty}{\tilde{F}_j^{\(k\)}\(x\)} = 0$ for all $k \ge 0$ and all $j$. This ensures that each $\tilde{F}_j^{\(k\)}$ is bounded and therefore that, for each $k \ge 0$, there is a constant $C_k$ such that $\norm{\tilde{F}_j^{\(k\)}}_{L^\infty\(\R\)} \le C_k$ for all $j$. It also means that we can define $F_j$ implicitly by $\tilde{F}_j = F_j \circ x^2$. Then the above analysis says that each $F_j \circ x^2$ is analytic, smooth, and, for each $k \ge 0$, there is a constant $C_k$ such that $\norm{\(F_j \circ x^2\)^{\(k\)}}_{L^\infty\(\R\)} \le C_k$ for all $j$. This says that, given an integer $k \ge 0$, for any multiindex $\beta$ of order $k$, if $w : \R^{1 + n} \to \R$ has the property that $\gamma \le \beta$ implies $\partial^\gamma w$ is bounded, then there is a number $C_k$ such that $\norm{\partial^\beta\[F_j \circ w^2\]}_{L^\infty\(\R^{1 + n}\)} \le C_k$ for all $j$.\end{remark}
We define the operator $F$ by the rule
\begin{equation}\label{F_def}F\(v\) \defined \frac{\phi_{< 1}}{1 + \(F_0 \circ \(rv\)^2\)v^2}\sumc{j}{1}{4}{\(F_j \circ \(rv\)^2\)\FFF_j\(v\)} + \phi_{> 1}r^{-2}v + r^{-1}\laplacian\phi + \phi_{> 1}r^{-1}\NNN\(rv + \phi\)\end{equation}
so that, upon replacing the $\tilde{F}_j$ with the $F_j$ and combining \eqref{box_v_with_cutoff_without_u}, \eqref{box_v_third_term_with_F_tilde_simplified}, and \eqref{F_def}, we can write the semilinear PDE
\begin{equation}\label{v_wave_equation}\square v = F\(v\)\period\end{equation}
We ultimately want to prove a global well-posedness theorem for \eqref{u_pde}. We do this by proxy, first proving one for \eqref{v_wave_equation}. The first step is to prove local well-posedness for $v$. For this we appeal to a classical result of H\"ormander. Before doing this, though, we establish our notation and basic toolbox of results.
\subsection{Definitions, Estimates, and Notation}\label{subsec:def_est_not}
\begin{remark}[Notation]\label{general_notation}The symbols $\N$ and $\Z_{\ge 0}$ will be used to denote the positive integers and the nonnegative integers, respectively.\end{remark}
\begin{remark}[More notation]\label{estimate_notation}The notation $a \lesssim_{p_1, p_2, \ldots , p_k} b$ will be used to denote $a \le Cb$ where $C$ is a constant that depends upon the parameters $p_1, p_2, \ldots , p_k$. We shall use the notation $a \gtrsim_{p_1, p_2, \ldots , p_k} b$ to mean $b \lesssim_{p_1, p_2, \ldots , p_k} a$. The notations $a \lesssim b$ and $a \gtrsim b$ are used to mean $a \le Cb$ and $b \le Ca$, respectively, where the constant $C$ depends only upon parameters that are considered fixed throughout this entire paper. Finally, the notation $a \simeq b$ will be used to indicate that both $a \lesssim b$ and $b \lesssim a$ are true.\end{remark}
\begin{remark}[Even more notation]\label{function_notation}In this paper we are concerned with the map $u$ (along with $v$ and $\Phi$, both of which will be introduced later), which can be viewed in two different ways. On one hand, it can be viewed as a function $\[0, T\] \times \R^2 \to \R$ of the two variables $t$ and $x$. On the other hand, it can be viewed as an evolution map in a space such as $C_t\(\[0, T\], H_{x, \mathrm{rad}}^s\(\R^2\)\)$. We shall have occasion to use both viewpoints throughout the course of this paper. In the former viewpoint, for a fixed $t \in \[0, T\]$, $u\(t, \cdot\)$ is a radial function of $x$; in the latter viewpoint, this would be expressed simply as $u\(t\)$. Similarly, the two notations $u\(t, x\)$ and $u\(t\)\(x\)$ may be used to mean the same thing. Similar statements hold for $v$ and $\Phi$.\end{remark}
\begin{lemma}[Sobolev embedding lemma]\label{sobolev_embedding_lemma}Let $\(k, n\) \in \Z_{\ge 0} \times \N$, and let $s > \frac{1}{2}n + k$. Then $H^s\(\R^n\) \into L^\infty\(\R^n\) \intersection C^k\(\R^n\)$ (the notation means $H^s\(\R^n\) \subseteq L^\infty\(\R^n\) \intersection C^k\(\R^n\)$ and the inclusion map is continuous). That is to say,
\begin{equation}\label{sobolev_embedding_lemma_estimate}f \in H^s\(\R^n\) \quad \Rightarrow \quad \underset{\order{\beta} \le k}{\sum}\norm{\partial^\beta f}_{L^\infty\(\R^n\)} \lesssim_s \norm{f}_{H^s\(\R^n\)}\period\end{equation}
\end{lemma}
\begin{proof}See the proof of Theorem $25$ in \cite{Selberg_-_Math_632}.\end{proof}
\begin{lemma}[Radial Sobolev embedding lemma]\label{radial_sobolev_embedding_lemma}Let $n \ge 2$ be an integer and $f \in H^1_{\mathrm{rad}}\(\R^n\)$. Then
\begin{equation}\label{radial_sobolev_embedding_lemma_r_estimate}\abs{f\(r\)} \lesssim_n \norm{f}_{H^1\(\R^n\)}r^{\frac{1}{2} - \frac{1}{2}n} \quad \mbox{for} \quad  \aev r \in \[1, \infty\)\period\end{equation}
\end{lemma}
\begin{proof}See the proof of Radial Lemma 1 of \cite{MR0454365}.\end{proof}
\begin{lemma}[Another radial Sobolev embedding lemma]\label{another_radial_sobolev_embedding_lemma}Let $n \ge 2$ be an integer, $s \in \(\frac{1}{2}, \frac{1}{2}n\)$, and $f \in H^s_{\mathrm{rad}}\(\R^n\)$. Then
\begin{equation}\label{another_radial_sobolev_embedding_lemma_r_estimate}\abs{f\(r\)} \lesssim_{n, s} \norm{f}_{\dot{H}^s\(\R^n\)}r^{s - \frac{1}{2}n} \quad \mbox{for} \quad  \aev r \in \[0, 1\)\period\end{equation}
\end{lemma}
\begin{proof}See the proof of Proposition 1 of \cite{MR2538202}.\end{proof}
\begin{defn}\label{function_spaces_defs}For each $s \in \R$, $I \subseteq \R$ an interval, and $n \in \N$, we define
\begin{align}
\label{X_s_I_n_def}X_{s, I, n} & \defined C_t\(I, H_{x, \mathrm{rad}}^s\(\R^n\)\) \intersection C_t^1\(I, H_{x, \mathrm{rad}}^{s - 1}\(\R^n\)\)\comma \\
\label{Y_s_I_n_def}Y_{s, I, n} & \defined \ob w \in L_t^\infty\(I, L_{x, \mathrm{rad}}^2\(\R^n\)\) : 0 \le j \le s \Rightarrow \partial_t^jw \in L_t^\infty\(I, H_x^{s - j}\(\R^n\)\)\cb\comma \quad \mbox{and} \\
\label{S_I_n_def}S_{I, n} & \defined I \times \R^n\period
\end{align}
The spaces $X_{s, I, n}$ and $Y_{s, I, n}$ are called our \textit{data spaces} and \textit{solution spaces}, respectively. We put the natural norms on $X_{s, I, n}$ and $Y_{s, I, n}$, given by
\begin{align}
\label{X_s_I_n_norm}& \norm{w}_{X_{s, I, n}} \defined \norm{\norm{w}_{H_x^s\(\R^n\)} + \norm{w_t}_{H_x^{s - 1}\(\R^n\)}}_{L^\infty\(I\)} \quad \mbox{and} \\
\label{Y_s_I_n_norm}& \norm{w}_{Y_{s, I, n}} \defined \sumc{j}{0}{\floor{s}}{\norm{\partial_t^jw}_{L_t^\infty\(I, H_x^{s - j}\(\R^n\)\)}}\comma
\end{align}
respectively. In this paper we shall always use the symbols $I$ and $I^*$ to denote the real intervals $\[0, T\)$ and $\[0, T^*\)$, respectively. The meanings of $T$ and $T^*$ should always be clear from the context.\end{defn}
\begin{defn}\label{indicator_functions_defn}Using the standard notation of an indicator function $\chi_S$ to describe the function whose value is $1$ exactly on the set $S$, we define the following three radial indicator functions $\R^4 \to \R$:
\begin{equation}\label{indicator_functions}\(\chi_0, \chi_1, \chi_\infty\) \defined \(\chi_{\ob r \le 1\cb}, \chi_{\ob\frac{1}{2} \le r \le 2\cb}, \chi_{\ob r \ge \frac{1}{2}\cb}\)\period\end{equation}
\end{defn}
\begin{defn}[Japanese bracket]\label{japanese_bracket_defn}We define the Japanese bracket operator $\jbracket{\cdot}$ in the usual way: by $\jbracket{\cdot} \defined \(1 + \abs{\cdot}^2\)^\frac{1}{2}$. Observe that $0 \le \jbracket{\cdot} \lesssim \chi_0 + \chi_\infty\abs{\cdot}$.\end{defn}
\begin{lemma}[Pointwise Estimates]\label{pointwise_lemma}Let $s > \frac{1}{2}$ and $w \in Y_{s, I, 4} \union H^s\(\R^4\) \union H^s\(S_{I, 4}\)$. Then we have a tetrachotomy of estimates as follows:
\begin{equation}\label{pointwise_lemma_estimates}\abs{w} \lesssim \ob\begin{array}{ll}\lesssim \chi_0r^{s - 2} + \chi_\infty\abs{w} & \frac{1}{2} < s < 1\comma \\ \lesssim \chi_0r^{s - 2} + \chi_\infty r^{-\frac{3}{2}} & 1 \le s < 2\comma \\ \lesssim_\epsilon \chi_0r^{-\epsilon} + \chi_\infty r^{-\frac{3}{2}} & s = 2 \quad \(\epsilon > 0\)\comma \\ \lesssim \jbracket{r}^{-\frac{3}{2}} & 2 < s \le \infty\period\end{array}\right.\end{equation}
\end{lemma}
\begin{proof}The first two of these estimates follow directly from the two radial Sobolev embedding estimates of \Cref{radial_sobolev_embedding_lemma,another_radial_sobolev_embedding_lemma}. The third one is at the non-included endpoint of \Cref{another_radial_sobolev_embedding_lemma}, and so we use \Cref{another_radial_sobolev_embedding_lemma} on $Y_{s - \epsilon, I, 4} \union H^{s - \epsilon}\(\R^4\) \union H^{s - \epsilon}\(S_{I, 4}\) \supseteq Y_{s, I, 4} \union H^s\(\R^4\) \union H^s\(S_{I, 4}\)$ for each $\epsilon > 0$. The last of the four estimates also makes use of \Cref{sobolev_embedding_lemma}.\end{proof}
We now introduce the Strichartz estimates. We introduce the general estimates, and then the specific ones that we shall use.
\begin{lemma}[General Strichartz Estimates]\label{strichartz_estimates_general}Suppose $S$ is an operator, $n \ge 2$ is an integer, $w_0, w_1 : \R^n \to \R$, and $w : S_{I, n} \to \R$ is radial and satisfies
\begin{equation}\label{strichartz_estimates_general_pde}\square w = S\(w\) \quad \mbox{and} \quad \(w\(0\), w_t\(0\)\) = \(w_0, w_1\)\end{equation}
on $S_{I, n}$. Then, if $p$, $q$, $\tilde{p}$, $\tilde{q}$, and $\gamma$ are chosen such that
\begin{equation}\label{strichartz_estimates_general_requirements}\begin{array}{lll}1 \le \tilde{p} \le 2 \le p \le \infty\comma & 1 < \tilde{q} \le 2 \le q < \infty\comma \\ \frac{1}{p} + \frac{n - 1}{q} < \frac{n - 1}{2} \ \mbox{or} \ \(p, q\) = \(\infty, 2\)\comma & \frac{1}{\tilde{p}} + \frac{n - 1}{\tilde{q}} > \frac{n + 1}{2} \ \mbox{or} \ \(\tilde{p}, \tilde{q}\) = \(1, 2\)\comma \\ \gamma \ge 0\comma & \frac{1}{p} + \frac{n}{q} = \frac{n}{2} - \gamma = \frac{1}{\tilde{p}} + \frac{n}{\tilde{q}} - 2\comma\end{array}\end{equation}
then $w$ satisfies the estimate
\begin{equation}\label{strichartz_estimates_general_estimate}\norm{w}_{L_t^p\(I, L_x^q\(\R^n\)\)} + \norm{w}_{C_t\(I, \dot{H}_x^\gamma\(\R^n\)\)} + \norm{w_t}_{C_t\(I, \dot{H}_x^{\gamma - 1}\(\R^n\)\)} \lesssim \norm{w_0}_{\dot{H}^\gamma\(\R^n\)} + \norm{w_1}_{\dot{H}^{\gamma - 1}\(\R^n\)} + \norm{S\(w\)}_{L_t^{\tilde{p}}\(I, L_x^{\tilde{q}}\(\R^n\)\)}\period\end{equation}
\end{lemma}
\begin{proof}This is proven in \cite{MR0512086}.\end{proof}
\begin{remark}\label{Strichartz_endpoints_remark}The requirement $\frac{1}{p} + \frac{n - 1}{q} < \frac{n - 1}{2}$ or $\(p, q\) = \(\infty, 2\)$ applies for radial functions on $\R^n$; it replaces the more restrictive general requirement of $\frac{1}{p} + \frac{n - 1}{2q} \le \frac{n - 1}{4}$ which applies for non-radial functions. A similar statement is true also for the requirement $\frac{1}{\tilde{p}} + \frac{n - 1}{\tilde{q}} > \frac{n + 1}{2}$ or $\(\tilde{p}, \tilde{q}\) = \(1, 2\)$. The details justifying this can be found in the proof of Theorem 1.5 of \cite{Sterbenz_-_2005}.\end{remark}
For our purposes we choose $n = 4$ and $\gamma = 1$. These two things in turn, after massaging the general inequalities, lead to the following Strichartz estimates.
\begin{lemma}[Strichartz Estimates]\label{strichartz_estimates_specific}Suppose $S$ is an operator, $w_0, w_1 : \R^4 \to \R$, and $w : S_{I, 4} \to \R$ is radial and satisfies
\begin{equation}\label{strichartz_estimates_pde}\square w = S\(w\) \quad \mbox{and} \quad \(w\(0\), w_t\(0\)\) = \(w_0, w_1\)\end{equation}
on $S_{I, 4}$. Then $w$ satisfies the estimate
\begin{equation}\label{strichartz_estimates_estimate}\norm{w}_{L_t^p\(I, L_x^\frac{4p}{p - 1}\(\R^4\)\)} + \norm{\gradient w}_{Y_{0, I, 4}} + \norm{w_t}_{Y_{0, I, 4}} \lesssim \norm{\gradient w_0}_{L^2\(\R^4\)} + \norm{w_1}_{L^2\(\R^4\)} + \norm{S\(w\)}_{L_t^1\(I, L_x^2\(\R^4\)\)}\comma \quad p \in \[2, \infty\]\period\end{equation}
\end{lemma}
\begin{proof}It follows from taking $p \in \[2, \infty\]$ and $\(n, q, \tilde{p}, \tilde{q}, \gamma\) = \(4, \frac{4p}{p - 1}, 1, 2, 1\)$ in \Cref{strichartz_estimates_general}.\end{proof}
\subsection{Local Well-Posedness}\label{subsec:lwp_v}
\begin{remark}\label{X_Y_embeddings}For $T > 0$, $n \in \N$, and $s > \frac{1}{2}n + 1$, each $w \in Y_{s, I, n}$ is bounded and continuously differentiable by the Sobolev embedding lemma. Therefore it can be extended at its boundary to form a member of $X_{s, \clos{I}, n}$. Using this extension, we shall regard, for such $T$, $n$, and $s$, $Y_{s, I, n} \into X_{s, \clos{I}, n}$. Similar reasoning shows that $s > \frac{1}{2}n$ implies  $X_{s, I, n}, Y_{s, I, n} \into L^\infty\(S_{\clos{I}, n}\)$.\end{remark}
\begin{theorem}[Local existence for $v$]\label{v_local_existence_theorem}Let
\begin{equation}\label{v_local_existence_hypotheses}s \ge 4 \quad \mbox{and} \quad \(v_0, v_1\) \in H_\mathrm{rad}^s\(\R^4\) \times H_\mathrm{rad}^{s - 1}\(\R^4\)\period\end{equation}
Then there is a $T > 0$ such that there is a $v \in X_{s, \clos{I}, 4}$ which solves the Cauchy problem
\begin{equation}\label{v_local_existence_pde}\square v = F\(v\)\comma \quad \(v\(0\), v_t\(0\)\) = \(v_0, v_1\)\end{equation}
on $S_{I, 4}$.\end{theorem}
\begin{theorem}[Continuation criterion for $v$]\label{v_continuation_criterion_theorem}Let $s$, $v_0$, and $v_1$ be as in the hypotheses of \Cref{v_local_existence_theorem} and let $T^*$ be the supremum among all $T > 0$ such that \eqref{v_local_existence_pde} has a solution $v \in X_{s, \clos{I}, 4}$ on $S_{I, 4}$. Then either $T^* = \infty$ (in which case our local solution is in fact a global solution) or $\(\jbracket{r}v, \jbracket{r}v_t, \jbracket{r}\gradient v\) \notin L^\infty\(S_{I^*, 4}\)$ (the notation means at least one of $\jbracket{r}v$, $\jbracket{r}v_t$, $\jbracket{r}\gradient v$ does not belong to $L^\infty\(S_{I^*, 4}\)$).\end{theorem}
\begin{proof}[Proof of \Cref{v_local_existence_theorem,v_continuation_criterion_theorem}]Together these two theorems comprise a slight modification of a result of L. H\"ormander, which is proven in Theorem 6.4.11 of \cite{MR1466700}.\end{proof}
It is natural to wonder at this point what all of these results about $v$ imply about $u$, which was our original object of study. Indeed, we only introduced $v$ to help solve a problem about $u$. The first simple thing to note in this connection is the following lemma.
\begin{lemma}\label{u_pde_vs_v_pde}Let $s \ge 4$, $T > 0$, and $v \in X_{s, \clos{I}, 4}$. Then $u$ satisfies \eqref{u_pde} on $S_{I, 2}$ if and only if $v$ satisfies \eqref{v_local_existence_pde} on $S_{I, 4}$.\end{lemma}
\begin{proof}The $\(\Rightarrow\)$ direction is immediate, since we derived \eqref{v_local_existence_pde} directly from the $\square u = \NNN\(u\)$ statement that is part of \eqref{u_pde}. For the $\(\Leftarrow\)$ direction, we assume $v$ satisfies \eqref{v_local_existence_pde}. Then, as before, we know $u$ satisfies $\square u = \NNN\(u\)$. Since $v \in X_{s, \clos{I}, 4}$, \Cref{X_Y_embeddings} tells us $v \in L^\infty\(S_{\clos{I}, 4}\)$. Hence, $u\(t, 0\) = rv\(t, 0\) + \phi\(0\) = \pi$. For large $r$, we use the fact that $v \in X_{s, \clos{I}, 4}$ implies $v \in H^1\(S_{\clos{I}, 4}\)$ and so, by \cref{pointwise_lemma}, $\abs{v} \lesssim r^{-\frac{3}{2}}$. It follows then that $\abs{u} \lesssim r^{-\frac{1}{2}}$, which implies $\fulllimit{r}{\infty}{u\(t, r\)} = 0$. Thus, $u$ satisfies \eqref{u_pde}.\end{proof}
\begin{lemma}\label{v_in_H^s_iff_u_in_H^s}Let $s \ge 4$ and $t \in I$. Then $v\(t\) \in H^s_\mathrm{rad}\(\R^4\) \iff u\(t\) \in H^s_\mathrm{rad}\(\R^2\)$.\end{lemma}
\begin{proof}This is the content of Lemma 1.3 of \cite{MR1278351}.\end{proof}
The upshot of \Cref{u_pde_vs_v_pde,v_in_H^s_iff_u_in_H^s} is that a $v \in X_{s, \clos{I}, 4}$ ($s \ge 4$) that solves \eqref{v_local_existence_pde} on $S_{I, 4}$ for some initial data $\(v_0, v_1\) \in H^s_{\mathrm{rad}}\(\R^4\) \times H^{s - 1}_{\mathrm{rad}}\(\R^4\)$ corresponds to a $u \in X_{s, \clos{I}, 2}$ that solves \eqref{u_pde} on $S_{I, 2}$ for corresponding initial data, with the boundary conditions coming for free.
\begin{theorem}[Local existence for $u$]\label{u_local_existence_theorem}Let
\begin{equation}\label{u_local_existence_hypotheses}s \ge 4 \quad \mbox{and} \quad \(u_0, u_1\) \in H_\mathrm{rad}^s\(\R^2\) \times H_\mathrm{rad}^{s - 1}\(\R^2\)\period\end{equation}
Then there is a $T > 0$ such that there is a $u \in X_{s, \clos{I}, 2}$ which solves the Cauchy problem
\begin{equation}\label{u_local_existence_pde}\square u = \NNN\(u\)\comma \quad \(u\(t, 0\), \fulllimit{r}{\infty}{u\(t, x\)}\) = \(\pi, 0\)\comma \quad \(u\(0\), u_t\(0\)\) = \(u_0, u_1\)\end{equation}
on $S_{I, 2}$.\end{theorem}
\begin{proof}Starting from \eqref{u_local_existence_pde} and letting $\(v_0, v_1\) \defined \(r^{-1}\(u_0 - \phi\), r^{-1}\(u_1 - \phi\)\)$, we arrive, making use of \Cref{v_in_H^s_iff_u_in_H^s}, at the hypotheses of \Cref{v_local_existence_theorem}. Thus we get a unique solution $v \in X_{s, \clos{I}, 4}$ on $S_{I, 4}$ to \eqref{v_local_existence_pde}. This $v$, by \Cref{u_pde_vs_v_pde}, gives rise to a $u \in X_{s, \clos{I}, 2}$ that solves \eqref{u_local_existence_pde} on $S_{I, 2}$. The uniqueness of this $u$ also follows from the uniqueness of $v$ in combination with \Cref{u_pde_vs_v_pde}.\end{proof}
\begin{theorem}[Global existence for $u$]\label{u_global_existence_theorem}Let $s$, $u_0$, and $u_1$ be as in the hypotheses of \Cref{u_local_existence_theorem}. Then, for all $T > 0$, there is a $u \in X_{s, \clos{I}, 2}$ that solves \eqref{u_local_existence_pde} on $S_{I, 2}$.\end{theorem}
\begin{remark}\label{u_gwp_remark}\Cref{u_global_existence_theorem} is the main goal of this paper. An outline of the strategy for obtaining this result is as follows. We shall prove this theorem in the special case of $s = 4$, since then the case $s > 4$ clearly follows by reduction to the $s = 4$ case. We begin with the problem for $u$ that we were initially interested in: \eqref{u_local_existence_pde}. We convert this into the new problem \eqref{v_local_existence_pde} for $v$. \Cref{v_local_existence_theorem} gives us a local solution for $v$, and we have its continuation criterion from \Cref{v_continuation_criterion_theorem}. If we can satisfy this criterion, then we have a global solution $v \in X_{s, \clos{I}, 4}$ on $S_{I, 4}$ to \eqref{v_local_existence_pde} for all $T > 0$. Then by \Cref{u_pde_vs_v_pde}, we also have a global solution $u \in X_{s, \clos{I}, 2}$ on $S_{I, 2}$ to \eqref{u_local_existence_pde} for all $T > 0$, our goal. Therefore, for the remainder of this paper, we proceed as follows. We define $T^*$ as in the statement of \Cref{v_continuation_criterion_theorem}. This leads to two possibilities: (i) $T^* = \infty$, and (ii) $T^* < \infty$. We want to prove that (i) is actually the case. In order to do this, we shall let $T > 0$ be given, and show that $T < T^*$. It will follow from this that $T^* = \infty$. The remainder of this paper is a proof that $T < \infty \Rightarrow T < T^*$. Thus, from here forth, we assume $T < \infty$ and therefore that $I$ is a finite interval.\end{remark}
\section{Derivation of $\Phi$ Function}\label{sec:der_Phi}
In this section we derive the ``change of variable'' function $\Phi$. We do this because the wave equation satisfied by $v$, \eqref{v_wave_equation}, is not amenable to Strichartz estimates, which are the main tools we wish to use in our analysis. To derive $\Phi$, we follow very closely the derivation of the $\Phi$ function from \cite{arXiv:1208.4977}. We will first introduce, in sequence, $\tilde{\Phi}_1$ , $\Phi_1$, and $\Phi_2$. Each of these is a step toward $\Phi$. Finally, we will introduce $\Phi$ as a modification of $\Phi_2$.
\subsection{Derivation of $\tilde{\Phi}_1$, $\Phi_1, \Phi_2$}\label{subsec:der_tilde_Phi_1_Phi_1_Phi_2}
We begin with a radial map $\tilde{\Phi}_1 : \R^{1 + 2} \to \R$, whose definition we do not yet specify. We write $\tilde{\Phi}_1 = \tilde{\Phi}_1\(z, \xi\) = \tilde{\Phi}_1\(z, \rho\)$, where $\rho = \abs{\xi}$. Next we define $\Phi_1 : \R^{1 + 2} \to \R$ by
\begin{equation}\label{Phi_1_def}\Phi_1 \defined \tilde{\Phi}_1 \circ \(u, r\)\period\end{equation}
After a short calculation we write $\square\Phi_1$ in terms of $\tilde{\Phi}_1$ as
\begin{equation}\label{box_Phi_1}\square\Phi_1 = \(u_t^2 - u_r^2\)\[\partial_{zz}\tilde{\Phi}_1 \circ \(u, r\)\] + \square u\[\partial_z\tilde{\Phi}_1 \circ \(u, r\)\] - \laplacian\tilde{\Phi}_1 \circ \(u, r\) - 2u_r\[\partial_{z\rho}\tilde{\Phi}_1 \circ \(u, r\)\]\period\end{equation}
Now, substituting \eqref{u_pde} and \eqref{NNN_def} into \eqref{box_Phi_1}, we obtain
\begin{equation}\label{box_Phi_1_calc}\square\Phi_1 = \ob\begin{array}{l}\(\partial_{zz}\tilde{\Phi}_1 \circ \(u, r\) - \frac{1}{2}\alpha^2r^{-2}A_1^{-1}\sin2u\[\partial_z\tilde{\Phi}_1 \circ \(u, r\)\]\)\(u_t^2 - u_r^2\) + \\ + \(-2\partial_{z\rho}\tilde{\Phi}_1 \circ \(u, r\) - 2r^{-1}\(1 - A_1^{-1}\)\[\partial_z\tilde{\Phi}_1 \circ \(u, r\)\]\)u_r - \\ - \frac{1}{2}r^{-2}A_1^{-1}\sin2u\[\partial_z\tilde{\Phi}_1 \circ \(u, r\)\] - \laplacian\tilde{\Phi}_1 \circ \(u, r\)\period\end{array}\right.\end{equation}
In order to complete the change of variable from $u$ to $\Phi_1$, we want the coefficients of the $\(u_t^2 - u_r^2\)$ and $u_r$ terms both to be zero. That is, we want to choose $\tilde{\Phi}_1$ to assure that
\begin{equation}\label{Phi_1_tilde_condition_1}\partial_{zz}\tilde{\Phi}_1 \circ \(u, r\) - \frac{1}{2}\alpha^2r^{-2}A_1^{-1}\sin2u\[\partial_z\tilde{\Phi}_1 \circ \(u, r\)\] = 0\end{equation}
and
\begin{equation}\label{Phi_1_tilde_condition_2}-2\partial_{z\rho}\tilde{\Phi}_1 \circ \(u, r\) - 2r^{-1}\(1 - A_1^{-1}\)\[\partial_z\tilde{\Phi}_1 \circ \(u, r\)\] = 0\period\end{equation}
An ODE analysis of \eqref{Phi_1_tilde_condition_1} shows that $\tilde{\Phi}_1$ must be of the form
\begin{equation}\label{Phi_1_tilde_form}\tilde{\Phi}_1 = q\int_\pi^z\(\rho^2 + \alpha^2\sin^2y\)^\frac{1}{2} \, dy\end{equation}
for some radial function $q = q\(\rho\)$. Taking this ansatz for $\tilde{\Phi}_1$ into \eqref{Phi_1_tilde_condition_2} shows that $q$ must be of the form $q\(\rho\) = C\rho^{-1}$ for a constant $C$. Taking $C = 1$, it follows that
\begin{equation}\label{Phi_1_tilde}\tilde{\Phi}_1 = \int_\pi^z\(1 + \alpha^2\rho^{-2}\sin^2y\)^\frac{1}{2} \, dy\period\end{equation}
One can see immediately from \eqref{Phi_1_tilde} that $\partial_z\tilde{\Phi}_1 \circ \(u, r\) = A_1^\frac{1}{2}$. Combining this with \eqref{Phi_1_tilde_condition_1}, \eqref{Phi_1_tilde_condition_2}, and \eqref{box_Phi_1_calc}, we obtain
\begin{equation}\label{box_Phi_1_calc_simplified}\square\Phi_1 = -\frac{1}{2}r^{-2}A_1^{-\frac{1}{2}}\sin 2u - \laplacian\tilde{\Phi}_1 \circ \(u, r\)\period\end{equation}
We want to write more convenient expressions for the two terms on the RHS of \eqref{box_Phi_1_calc_simplified}. For this we first define
\begin{equation}\label{A_2_A_3_def}A_2 = A_2\(y, \rho\) \defined 1 + \alpha^2\rho^{-2}\sin^2y \quad \mbox{and} \quad A_3 = A_3\(y, r\) \defined 1 + \alpha^2r^{-2}\sin^2y\period\end{equation}
Thus we can rewrite \eqref{Phi_1_tilde} as $\tilde{\Phi}_1 = \int_\pi^z\(A_2\)^\frac{1}{2} \, dy$ and therefore
\begin{equation}\label{Phi_1_with_A_3}\Phi_1 = \int_\pi^uA_3^\frac{1}{2} \, dy\period\end{equation}
A direct calculation then shows
\begin{equation}\label{laplacian_Phi_1_tilde_calc}\laplacian\tilde{\Phi}_1 = \rho^{-2}\int_\pi^zA_2^{-\frac{3}{2}}\(A_2^2 - 1\) \, dy\period\end{equation}
Therefore
\begin{equation}\label{laplacian_Phi_1_tilde_simplified}\laplacian\tilde{\Phi}_1 \circ \(u, r\) = r^{-2}\int_\pi^uA_3^{-\frac{3}{2}}\(A_3^2 - 1\) \, dy\period\end{equation}
This gives us a handle on the second term on the RHS of \eqref{box_Phi_1_calc_simplified}. For the first term, we make a direct calculation using the fundamental theorem of calculus to obtain
\begin{equation}\label{first_term_ftoc_simplified_2}-\frac{1}{2}r^{-2}A_1^{-\frac{1}{2}}\sin 2u = -r^{-2}\int_\pi^uA_3^{-\frac{3}{2}}\[1 - \alpha^{-2}r^2\(A_3^2 - 1\)\] \, dy\period\end{equation}
We now insert \eqref{laplacian_Phi_1_tilde_simplified} and \eqref{first_term_ftoc_simplified_2} into \eqref{box_Phi_1_calc_simplified} to obtain
\begin{equation}\label{box_Phi_1_simplified}\square\Phi_1 = \(\alpha^{-2} - r^{-2}\)\Phi_1 - \alpha^{-2}\int_\pi^uA_3^{-\frac{3}{2}} \, dy\period\end{equation}
We still see an $r^{-2}$-type singularity in \eqref{box_Phi_1_simplified}, similar to that in \eqref{u_pde}. To try to resolve it, we shall take an analogous approach. For $u$ we made two transformations: multiplication by $r^{-1}$ and subtraction of the cutoff function $\phi$. We presently do something analogous to the first of these two transformations to $\Phi_1$. To this end we define $\Phi_2 : \R^{1  + 4} \to \R$ by
\begin{equation}\label{Phi_2_def}\Phi_2 \defined r^{-1}\Phi_1\period\end{equation}
From here we lift again to dimension $\R^{1 + 4}$ to write the following wave equation for $\Phi_2$:
\begin{equation}\label{box_Phi_2_substituted}\square\Phi_2 = \alpha^{-2}\(\Phi_2 - r^{-1}\int_\pi^uA_3^{-\frac{3}{2}} \, dy\)\period\end{equation}
We now try to address the $r^{-1}$ factor in the second term of \eqref{box_Phi_2_substituted}. As a first step, we define
\begin{equation}\label{A_4_def}A_4 = A_4\(y, r\) \defined 1 + \alpha^2r^{-2}\sin^2ry\comma\end{equation}
the definition in the limiting case $r = 0$ being the natural one. It is clear that $A_4$ is smooth. Recalling our $\phi_{< 1}$ cutoff function first introduced immediately before \eqref{box_v_with_cutoff_without_u}, we aim to split the second term of \eqref{box_Phi_2_substituted} into two parts: one for small $r$ and one for large $r$. For the small $r$ part, we make the change of variable $y \mapsto ry + \pi$ in the integral, so that we get a factor of $r$ in the differential that cancels the $r^{-1}$ outside the integral. The large $r$ part we leave alone. After doing this and making some simplifications (using the facts that $\phi = \pi$ on the support of $\phi_{< 1}$ and $\sin^2$ has a period of $\pi$), we rewrite \eqref{box_Phi_2_substituted} as
\begin{equation}\label{box_Phi_2_simplified}\square\Phi_2 = \alpha^{-2}\(\Phi_2 - \phi_{< 1}\int_0^vA_4^{-\frac{3}{2}} \, dy - \phi_{> 1}r^{-1}\int_\pi^uA_3^{-\frac{3}{2}} \, dy\)\period\end{equation}
\subsection{Derivation of $\Phi$}\label{subsec:der_Phi}
Unfortunately this is still not satisfactory. By using the definition of $\Phi_2$ along with our simplified expression for $\Phi_1$ and \Cref{sobolev_embedding_lemma,radial_sobolev_embedding_lemma}, we see that the best estimate for $\Phi_2$ that we can obtain so far is, for fixed $t \in I$, $\abs{\Phi_2\(t\)} \lesssim r^{-\frac{1}{2}}$. This is far from enough to put $\Phi_2\(t\)$ in $L^2\(\R^4\)$. In view of this, we make one more transformation whose purpose is to eliminate this problem. We define
\begin{equation}\label{Phi_def}\Phi \defined \Phi_2 + \phi_{> 1}r^{-1}\int_0^\pi A_3^{-\frac{3}{2}} \, dy\period\end{equation}
If we then define
\begin{equation}\label{I_j_def}I_1 \defined \phi_{< 1}r^{-1}\int_\pi^uA_3^\frac{1}{2} \, dy\comma \quad I_2 \defined \phi_{> 1}r^{-1}\int_0^uA_3^\frac{1}{2} \, dy\comma \quad \mbox{and} \quad I_3 \defined -\phi_{> 1}r^{-1}A_3^{-\frac{1}{2}}\int_0^\pi\(A_3 - A_3^{-1}\) \, dy\comma\end{equation}
then a direct calculation shows
\begin{equation}\label{Phi_equals_sum_I_j}\Phi = I_1 + I_2 + I_3\period\end{equation}
We now want to obtain estimates for $I_1$, $I_2$, and $I_3$. For $I_1$, we first define
\begin{equation}\label{A_5_def}A_5 = A_5\(y, r\) \defined 1 + \alpha^2r^{-2}\sin^2\(ry + \phi\)\comma\end{equation}
the definition in the limiting case $r = 0$ being the natural one, and note that $A_5$ is smooth. Then we obtain, after a change of variable $y \mapsto ry + \phi$,
\begin{equation}\label{I_1_estimate}I_1 = \phi_{< 1}\int_0^vA_5^\frac{1}{2} \, dy\period\end{equation}
For $I_2$, the same change of variable and a subsequent rearrangement gives
\begin{equation}\label{I_2_preliminary_estimate}I_2 = \phi_{> 1}\int_0^vA_5^\frac{1}{2} \, dy + \phi_{> 1}r^{-1}\int_0^\phi A_3^\frac{1}{2} \, dy\period\end{equation}
In observing the second term of \eqref{I_2_preliminary_estimate}, one notes that the $\phi_{> 1}$ factor implies that the term is supported only for $r \ge \frac{1}{2}$, but the upper bound on the integral of $\phi$ implies that the term is also supported only for $r \le 2$. This implies the $r^{-1}$ term in front is harmless as well, and we may rewrite this second term simply as $\phi_{\sim 1}$, a nonnegative, smooth cutoff function localized to $\frac{1}{2} \le r \le 2$. Therefore we rewrite \eqref{I_2_preliminary_estimate} as
\begin{equation}\label{I_2_estimate}I_2 = \phi_{> 1}\int_0^vA_5^\frac{1}{2} \, dy + \phi_{\sim 1}\period\end{equation}
Now, for $I_3$, we first compute directly that $A_3 - A_3^{-1} \lesssim r^{-2}$ for $r \ge \frac{1}{2}$. Furthermore, $A_3^{-\frac{1}{2}} \le 1$. Thus, also using the fact that $A_3 - A_3^{-1} \ge 0$, we can rewrite $I_3$ in terms of $\phi_{\ge \frac{1}{2}}$, a nonnegative, smooth, bounded cutoff function localized to $r \ge \frac{1}{2}$, as
\begin{equation}\label{I_3_estimate}I_3 = \phi_{\ge \frac{1}{2}}r^{-3}\period\end{equation}
Putting together \eqref{Phi_equals_sum_I_j}, \eqref{I_1_estimate}, \eqref{I_2_estimate}, and \eqref{I_3_estimate}, absorbing $\phi_{\sim 1}$ into $\phi_{\ge \frac{1}{2}}$, and simplifying, we have shown
\begin{equation}\label{Phi_H0}\Phi = \int_0^vA_5^\frac{1}{2} \, dy + \phi_{\ge \frac{1}{2}}r^{-3}\period\end{equation}
This equation \eqref{Phi_H0} is a direct expression for $\Phi$. However, we want to relate $\Phi$ to $\Phi_2$ in a manner such that we can use our $\Phi_2$ analysis to give us information about $\square\Phi$. Thus, going back to the definition \eqref{Phi_def} of $\Phi$, we see that we can write instantly
\begin{equation}\label{Phi_estimate}\Phi = \Phi_2 + \psi_{\ge \frac{1}{2}}r^{-1}\comma\end{equation}
where $\psi_{\ge \frac{1}{2}}$ is a nonnegative, smooth, bounded cutoff function localized to $r \ge \frac{1}{2}$. A lengthy but direct calculation shows that $\square\(\psi_{\ge \frac{1}{2}}r^{-1}\)$ can be written as the product of a smooth, bounded function supported on $r \ge \frac{1}{2}$ and $r^{-3}$. For this reason, we shall write $\square\(\psi_{\ge \frac{1}{2}}r^{-1}\) = \phi_{\ge \frac{1}{2}}r^{-3}$. Thus, taking $\square$ of both sides of \eqref{Phi_estimate}, appealing to \eqref{box_Phi_2_simplified}, and rearranging a bit, we obtain
\begin{equation}\label{box_Phi_prelim}\square\Phi = \alpha^{-2}\(\Phi - \phi_{> 1}r^{-1}\int_0^uA_3^{-\frac{3}{2}} \, dy - \phi_{< 1}\int_0^vA_4^{-\frac{3}{2}} \, dy\) + \phi_{\ge \frac{1}{2}}r^{-3}\period\end{equation}
Simple computations show
\begin{equation}\label{A_j_computations}\phi_{< 1}\int_0^vA_4^{-\frac{3}{2}} \, dy = \phi_{< 1}\int_0^vA_5^{-\frac{3}{2}} \, dy \quad \mbox{and} \quad \phi_{> 1}\int_0^uA_3^{-\frac{3}{2}} \, dy = \phi_{> 1}r\int_{-\frac{\phi}{r}}^0A_5^{-\frac{3}{2}} \, dy + \phi_{> 1}r\int_0^vA_5^{-\frac{3}{2}} \, dy\period\end{equation}
The first term on the RHS of the second member of \eqref{A_j_computations} is smooth, bounded, and supported only when $\frac{1}{2} \le r \le 2$, so it can be absorbed into the last term on the RHS of \eqref{box_Phi_prelim}. If we do this, substitute \eqref{A_j_computations} into \eqref{box_Phi_prelim}, and simplify, we obtain
\begin{equation}\label{Phi_wave_equation}\square\Phi = \alpha^{-2}\(\Phi - \int_0^vA_5^{-\frac{3}{2}} \, dy\) + \phi_{\ge \frac{1}{2}}r^{-3}\period\end{equation}
We shall analyze \eqref{Phi_wave_equation} and \eqref{Phi_H0} in the next section.
\begin{remark}\label{phi_ge_1_2_properties}This function $\phi_{\ge \frac{1}{2}}r^{-3}$ has additional properties worth noting. If we look at its definition, and observe that, for $r \ge 1$ and $j\in \Z_{\ge 0}$,
\begin{equation}\label{phi_ge_1_2_properties_estimates}\partial_r^j\(r^{-1}\) \lesssim r^{-\(j + 1\)}\comma \quad \partial_r^j\(\phi_{> 1}\) \lesssim \ob\begin{array}{ll}1 & j = 0\comma \\ 0 & j > 0\comma\end{array}\right. \quad \mbox{and} \quad \partial_r^j\(\int_0^\pi\(A_3 - A_3^{-1}\) \, dy\) \lesssim \ob\begin{array}{ll}r^{-2} & j = 0\comma \\ r^{-3} & j > 0\comma\end{array}\right.\end{equation}
we see immediately that not only do we have the estimate $\abs{\phi_{\ge \frac{1}{2}}r^{-3}} \lesssim \chi_\infty r^{-3}$, but we have the same estimate for all of its derivatives as well. Thus, from here forth, we shall make use freely of the estimates
\begin{equation}\label{phi_ge_1_2_estimate}\abs{\gradient^j\(\phi_{\ge \frac{1}{2}}r^{-3}\)}, \abs{\partial_r^j\(\phi_{\ge \frac{1}{2}}r^{-3}\)} \lesssim \chi_\infty r^{-3}\comma \quad j \in \Z_{\ge 0}\period\end{equation}\end{remark}
\begin{remark}[Simplified notation]Before we begin the $Y_{1, I, 4}$ analysis of $\Phi$, we introduce some streamlined notation that will be used for the remainder of this paper. From here forth, we shall use
\begin{alignat}{3}
\label{Y_s_notation}& Y_s \quad & \quad \mbox{as a shorthand notation for} \quad & \quad Y_{s, I, 4}\comma \\
\label{t_function_norm_notation}& \norm{f}_{L^p} \quad & \quad \mbox{as a shorthand notation for} \quad & \quad \norm{f}_{L^p\(I\)} \quad \mbox{(if $f$ is a pure function of $t$)}\comma \\
\label{x_function_norm_notation}& \norm{f}_{L^q} \quad & \quad \mbox{as a shorthand notation for} \quad & \quad \norm{f}_{L^q\(\R^4\)} \quad \mbox{(if $f$ is a pure function of $x$)}\comma \\
\label{t_x_x_function_norm_notation}& \norm{f}_{L_x^q} \quad & \quad \mbox{as a shorthand notation for} \quad & \quad \norm{f}_{L_x^q\(\R^4\)} \quad \mbox{(if $f$ is a function of $t$ and $x$)}\comma \\
\label{t_x_function_norm_notation}& \norm{f}_{L^pL^q} \quad & \quad \mbox{as a shorthand notation for} \quad & \quad \norm{f}_{L_t^p\(I, L_x^q\(\R^4\)\)} \quad \mbox{(if $f$ is a function of $t$ and $x$)}\comma \\
\label{y_x_function_norm_notation}& \norm{f}_{L^pL^q} \quad & \quad \mbox{as a shorthand notation for} \quad & \quad \norm{f}_{L_y^p\(I, L_x^q\(\R^4\)\)} \quad \mbox{(if $f$ is a function of $y$ and $x$)}\period
\end{alignat}
\end{remark}
\section{$Y_1$ Analysis of $\Phi$}\label{sec:Y_1}
Let us define $s$, $v_0$, and $v_1$, as in the statement of \Cref{v_continuation_criterion_theorem} now and for the rest of this paper. We ultimately are interested in showing that the continuation criterion from \Cref{v_continuation_criterion_theorem},
\begin{equation}\label{continuation_criterion_goal}\jbracket{r}v, \jbracket{r}v_t, \jbracket{r}\gradient v \in L^\infty\(S_{I, 4}\)\comma\end{equation}
is satisfied. Eventually we achieve this by placing $\Phi \in Y_4$. A first step toward this goal is to place $\Phi \in Y_1$. In order to do this, we must control the $Y_0$ norms of $\Phi$, $\Phi_t$, and $\gradient\Phi$. Before doing this, let us define the energy operator $E$ associated with solutions $u$ of \eqref{u_pde}. This energy is denoted by $Eu : I \to \R$ (or $Eu\(t\)$) and defined by
\begin{equation}\label{energy_def}Eu \defined \frac{1}{2}\int_0^\infty\[A_1\(u_t^2 + u_r^2\) + r^{-2}\sin^2u\]r \, dr\period\end{equation}
\begin{lemma}\label{conservation_of_energy}If $u$ satisfies the hypotheses of \Cref{u_local_existence_theorem}, then $Eu$ is constant.\end{lemma}
\begin{proof}There are several proofs of this important fact. A particularly transparent one is to show $\(Eu\)^\prime = 0$ pointwise. To do this, one fixes $t \in I$ and picks $\tilde{T} \in \[t, T\)$ for which a local solution $u$ exists (for example, $\tilde{T} = \frac{1}{2}\(t + T\)$ works). When one takes the derivative of \eqref{energy_def} and subsequently integrates by parts, one sees the difference, $\square u - \NNN\(u\)$, of the two sides of the equation \eqref{u_local_existence_pde} under the integral sign, meaning $\(Eu\)^\prime\(t\) = 0$. Thus, one is only left to check, for example, that $Eu\(0\)$ is finite. But this comes immediately from the initial conditions in \eqref{u_local_existence_pde}.\end{proof}
\begin{prop}\label{Phi_t_in_Y_0}$\Phi_t \in Y_0$.\end{prop}
\begin{proof}First observe from \eqref{Phi_H0} that
\begin{equation}\label{Phi_t_calc}\Phi_t = r^{-1}A_1^\frac{1}{2}u_t\period\end{equation}
Applying \eqref{Phi_t_calc}, \eqref{energy_def}, and Lemma \ref{conservation_of_energy}, we obtain
\begin{equation}\label{Phi_t_L_x^2_estimate}\norm{\Phi_t}_{Y_0}^2 = \norm{\int_{\R^4}r^{-2}A_1u_t^2 \, dx}_{L^\infty} \simeq \norm{\int_0^\infty A_1u_t^2r \, dr}_{L^\infty} \le 2\norm{Eu}_{L^\infty} = 2Eu\(0\) < \infty\period\end{equation}
\end{proof}
\begin{lemma}\label{Phi_L^2_estimate_lemma}For each $t \in I$, $\norm{\Phi\(t\)}_{L^2} \lesssim t + 1$.\end{lemma}
\begin{proof}We begin with a basic estimate for $\Phi$. From \eqref{Phi_H0} we have
\begin{equation}\label{Phi_improved_estimate}\abs{\Phi} \lesssim \abs{v} + v^2 + \abs{\phi_{\ge \frac{1}{2}}}r^{-3}\period\end{equation}
Specializing now to the case $t = 0$ and appealing to the Sobolev embedding lemma and the regularity of the initial data, \eqref{Phi_improved_estimate} implies
\begin{equation}\label{Phi_0_L^2_estimate}\norm{\Phi\(0\)}_{L^2} \lesssim \norm{v\(0\)}_{L^2} +  \norm{v^2\(0\)}_{L^2} + \norm{\phi_{\ge \frac{1}{2}}r^{-3}}_{L^2} \lesssim 1\period\end{equation}
We shall use \eqref{Phi_0_L^2_estimate} in conjunction with the mean value theorem to finish the proof of this lemma. Let $t \in I$ be given. By the mean value theorem, for each $r > 0$, there is a $t_r \in \[0, t\]$ such that
\begin{equation}\label{Phi_MVT_estimate_rearranged}\Phi\(t, r\) = t\Phi_t\(t_r, r\) - \Phi\(0, r\)\period\end{equation}
It follows from \eqref{Phi_MVT_estimate_rearranged} and \eqref{Phi_0_L^2_estimate} that
\begin{equation}\label{Phi_MVT_L^2_estimate}\norm{\Phi\(t\)}_{L^2} \lesssim t\norm{\Phi_t\(t_r\)}_{L^2} + 1\period\end{equation}
Applying \Cref{Phi_t_in_Y_0} to \eqref{Phi_MVT_L^2_estimate} gives
\begin{equation}\label{Phi_L^2_estimate}\norm{\Phi\(t\)}_{L^2} \lesssim t + 1\period\end{equation}
\end{proof}
\begin{cor}\label{Phi_in_Y_0}$\Phi \in Y_0$.\end{cor}
\begin{proof}It follows immediately from Lemma \ref{Phi_L^2_estimate_lemma} and the hypothesis that $T < \infty$.\end{proof}
\begin{prop}\label{Phi_in_Y_1}$\Phi \in Y_1$.\end{prop}
\begin{proof}By dint of \Cref{Phi_t_in_Y_0,Phi_in_Y_0}, it suffices to show $\gradient\Phi \in Y_0$. For this, we start by multiplying \eqref{Phi_wave_equation} by $\Phi_t$ to obtain
\begin{equation}\label{box_4_Phi_Phi_t_calc}\square\Phi\Phi_t = \Phi_t\[\alpha^{-2}\(\Phi - \int_0^vA_5^{-\frac{3}{2}} \, dy\) + \phi_{\ge \frac{1}{2}}r^{-3}\]\period\end{equation}
Next we integrate both sides of \eqref{box_4_Phi_Phi_t_calc} over $\R^4$:
\begin{equation}\label{int_box_4_Phi_Phi_t_calc}\int_{\R^4}\square\Phi\Phi_t \, dx = \int_{\R^4}\Phi_t\[\alpha^{-2}\(\Phi - \int_0^vA_5^{-\frac{3}{2}} \, dy\) + \phi_{\ge \frac{1}{2}}r^{-3}\] \, dx\period\end{equation}
Using the divergence theorem on the LHS of \eqref{int_box_4_Phi_Phi_t_calc} and simplifying, we obtain
\begin{equation}\label{int_box_4_Phi_Phi_t_divergence}\int_{\R^4}\square\Phi\Phi_t \, dx = \int_{\R^4}\(\Phi_{tt}\Phi_t - \gradient\Phi \cdot \gradient\Phi_t\) \, dx = \frac{1}{2}\partial_t\int_{\R^4}\(\Phi_t^2 + \abs{\gradient\Phi}^2\) \, dx\period\end{equation}
Putting \eqref{int_box_4_Phi_Phi_t_calc} and \eqref{int_box_4_Phi_Phi_t_divergence} together, we then have
\begin{equation}\label{dt_Phi_t_Phi_r_L_2^2}\partial_t\int_{\R^4}\(\Phi_t^2 + \abs{\gradient\Phi}^2\) \, dx = 2\int_{\R^4}\Phi_t\[\alpha^{-2}\(\Phi - \int_0^vA_5^{-\frac{3}{2}} \, dy\) + \phi_{\ge \frac{1}{2}}r^{-3}\] \, dx\period\end{equation}
Thus, by the Cauchy-Schwarz and Minkowski inequalities,
\begin{equation}\label{dt_Phi_t_Phi_r_L_2^2_C-S_estimate}\abs{\partial_t\int_{\R^4}\(\Phi_t^2 + \abs{\gradient\Phi}^2\) \, dx} \lesssim \norm{\Phi_t}_{L_x^2}\(\norm{\Phi}_{L_x^2} + \norm{\int_0^vA_5^{-\frac{3}{2}} \, dy}_{L_x^2} + \norm{\phi_{\ge \frac{1}{2}}r^{-3}}_{L_x^2}\)\period\end{equation}\normalsize
Obviously the third term in parentheses on the RHS of \eqref{dt_Phi_t_Phi_r_L_2^2_C-S_estimate} is $\lesssim 1$. Combining this observation with \Cref{Phi_t_in_Y_0} and \Cref{Phi_in_Y_0} gives
\begin{equation}\label{dt_Phi_t_Phi_r_L_2^2_C-S_estimate_improved}\abs{\partial_t\int_{\R^4}\(\Phi_t^2 + \abs{\gradient\Phi}^2\) \, dx} \lesssim 1 + \norm{\int_0^vA_5^{-\frac{3}{2}} \, dy}_{L_x^2}\period\end{equation}
Now, since $A_5 \ge 1$, we can again make use of \Cref{Phi_in_Y_0} and estimate the integral on the RHS as
\begin{equation}\label{A_5_integral_estimate}\norm{\int_0^vA_5^{-\frac{3}{2}} \, dy}_{L_x^2} \le \norm{\int_0^vA_5^\frac{1}{2} \, dy}_{L_x^2} \le \norm{\Phi}_{L_x^2} + \norm{r^{-3}\phi_{\ge \frac{1}{2}}}_{L^2} \lesssim 1\period\end{equation}
It follows from \eqref{dt_Phi_t_Phi_r_L_2^2_C-S_estimate_improved} and \eqref{A_5_integral_estimate} that
\begin{equation}\label{dt_Phi_t_Phi_r_L_2^2_C-S_estimate_final}\partial_t\int_{\R^4}\(\Phi_t^2 + \abs{\gradient\Phi}^2\) \, dx \lesssim 1\period\end{equation}
From here we would like to integrate both sides over the interval $I$ to reach the desired conclusion. In order to do this, in view of \Cref{Phi_t_in_Y_0}, it suffices to show $\norm{\gradient\Phi\(0\)}_{L^2} \lesssim 1$. To do this we take the gradient of both sides of \eqref{Phi_H0} to obtain the estimate
\begin{equation}\label{gradient_Phi_0_1}\gradient\Phi\(0\) = A_1^\frac{1}{2}\(0\)\gradient v\(0\) + \frac{1}{2}\int_0^{v\(0\)}A_5^{-\frac{1}{2}}\gradient A_5 \, dy - r^{-3}\phi_{\ge \frac{1}{2}}\period\end{equation}
From here we note that the initial conditions imply $A_1^\frac{1}{2}\(0\) \lesssim 1$. Additionally we know $A_5^{-\frac{1}{2}} \lesssim 1$, and a Maclaurin analysis of $\gradient A_5$ shows $\abs{\gradient A_5} \lesssim \chi_0ry^4 + \chi_\infty r^{-2}$. It follows immediately that we can upgrade \eqref{gradient_Phi_0_1} to
\begin{equation}\label{gradient_Phi_0_2}\abs{\gradient\Phi\(0\)} \lesssim \abs{\gradient v\(0\)} + \abs{v\(0\)} + \abs{r^{-3}\phi_{\ge \frac{1}{2}}}\period\end{equation}
Wielding this estimate and integrating both sides of \eqref{dt_Phi_t_Phi_r_L_2^2_C-S_estimate_final} over the interval $I$ then gives
\begin{equation}\label{Phi_t_Phi_r_L_2^2_estimate}\int_{\R^4}\(\Phi_t^2 + \abs{\gradient\Phi}^2\) \, dx \lesssim 1\comma\end{equation}
the suppressed constant depending on the finite number $T$, but not $t$. Now dropping the $\Phi_t$ term and rewriting gives
\begin{equation}\label{Phi_r_L^2_estimate}\norm{\gradient\Phi}_{L_x^2} \lesssim 1\period\end{equation}
This proves $\gradient\Phi \in Y_0$ and completes the proof of the proposition.\end{proof}
\section{$Y_2$ Analysis of $\Phi$}\label{sec:Y_2}
In this section we upgrade the regularity of $\Phi$ to $Y_2$. In order to do this, we need to show $\laplacian\Phi, \gradient\Phi_t, \Phi_{tt} \in Y_0$. Our approach in this section is as follows. In the first subsection we introduce the Strichartz estimate that we eventually prove. In the next three subsections we derive the three estimates needed to prove the Strichartz estimate from the first subsection holds. This Strichartz estimate guarantees $\gradient\Phi_t, \Phi_{tt} \in Y_0$. Finally, in the fifth subsection, we use this information along with the wave equation \eqref{Phi_wave_equation} to prove $\laplacian\Phi \in Y_0$ and then conclude $\Phi \in Y_2$.
\subsection{Strichartz Estimate for $\Phi_t$}\label{subsec:Y_2_strichartz_estimate}
We first observe that we can relate $v_t$ to $\Phi_t$ by differentiating \eqref{Phi_H0} with respect to $t$ to obtain
\begin{equation}\label{v_t_calc_H2}v_t = A_1^{-\frac{1}{2}}\Phi_t\period\end{equation}
Now, taking the derivative of \eqref{Phi_wave_equation} with respect to $t$ and appealing to \eqref{v_t_calc_H2}, we obtain the following wave equation for $\Phi_t$:
\begin{equation}\label{Phi_t_wave_equation}\square\Phi_t = \alpha^{-2}\(1 - A_1^{-2}\)\Phi_t\period\end{equation}
Therefore we have the following Strichartz estimate for $\Phi_t$:
\begin{equation}\label{Phi_t_strichartz_estimate}\norm{\Phi_t}_{L^pL^\frac{4p}{p - 1}} + \norm{\gradient\Phi_t}_{Y_0} + \norm{\Phi_{tt}}_{Y_0} \lesssim \norm{\gradient\Phi_t\(0\)}_{L^2} + \norm{\Phi_{tt}\(0\)}_{L^2} + \norm{\(1 - A_1^{-2}\)\Phi_t}_{L^1L^2}\comma \quad p \in \[2, \infty\]\period\end{equation}
\subsection{Estimate for $\norm{\(1 - A_1^{-2}\)\Phi_t}_{L^1L^2}$}\label{subsec:Y_2_strichartz_estimate_third_term}
This estimate can be obtained very quickly, but in this subsection we develop our estimates more methodically, as they will help us later in this paper. First, clearly $A_1$ satisfies
\begin{equation}\label{A_1^-sigma_H2}0 \le A_1^{-\sigma} \le 1\comma \quad A_1^{-\sigma} \in L^pL^\infty\comma \, \(p, \sigma\) \in \[1, \infty\] \times \[0, \infty\]\period\end{equation}
Therefore, for $\sigma \in \N$,
\begin{equation}\label{1_-_A_1^-sigma_calc_H2}0 \le 1 - A_1^{-\sigma} = A_1^{-\sigma}\(A_1^\sigma - 1\) = A_1^{-\sigma}\(A_1 - 1\)\sumc{j}{0}{\sigma - 1}{A_1^j} \lesssim_\sigma A_1 - 1\period\end{equation}
Since $\Phi \in Y_1$, we have, interpolating spaces as necessary,
\begin{equation}\label{Phi_H2}\abs{\Phi} \lesssim \chi_0r^{-1} + \chi_\infty r^{-\frac{3}{2}}\comma \quad \Phi \in L^pL^q\comma \, \(p, q\) \in \[1, \infty\] \times \[2, 4\)\period\end{equation}
Clearly $A_5$ satisfies
\begin{equation}\label{A_5^-sigma_H2}0 \le A_5^{-\sigma} \le 1 \le A_5^\sigma \le \infty\comma \quad A_5^{-\sigma} \in L^pL^\infty\comma \, \(p, \sigma\) \in \[1, \infty\] \times \[0, \infty\]\period\end{equation}
Now, relating $v$ to $\Phi$ via \eqref{Phi_H0} and applying \eqref{Phi_H2} and \eqref{A_5^-sigma_H2}, we obtain
\begin{equation}\label{v_H2}\abs{v} \lesssim \abs{\Phi} + \chi_\infty r^{-3} \lesssim \chi_0r^{-1} + \chi_\infty r^{-\frac{3}{2}}\comma \quad v \in L^pL^q\comma \, \(p, q\) \in \[1, \infty\] \times \[2, 4\)\period\end{equation}
This means, in particular, that $\abs{v} \lesssim r^{-1}$. This is a useful estimate because we shall have occasion frequently to estimate terms such as $r^{-\sigma}\abs{\sin\(rv + \phi\)}^\sigma$ for $\sigma \ge 0$. In general, noting that $\phi = \pi$ for $r \le 1$, this can be estimated only as $r^{-\sigma}\min\ob r^\sigma\abs{v}^\sigma, 1\cb = \min\ob\abs{v}^\sigma, r^{-\sigma}\cb$. However, because of our $v$ estimate, we can now always rely upon $r^{-\sigma}\abs{\sin\(rv + \phi\)}^\sigma \lesssim\abs{v}^\sigma$. From here forth we make use of this type of estimate without further remark. We also have, by our initial conditions,
\begin{equation}\label{v_0_H2}\abs{v\(0\)} \lesssim \jbracket{r}^{-\frac{3}{2}}\comma \quad v\(0\) \in L^q\comma \, q \in \[2, \infty\]\period\end{equation}
Thus, using the above remark, \eqref{v_H2}, and \eqref{v_0_H2}, we deduce
\begin{equation}\label{A_1_-_1_H2}\ob\begin{array}{ll}0 \le A_1 - 1 \lesssim v^2 \lesssim \chi_0r^{-2} + \chi_\infty r^{-3}\comma & A_1 - 1 \in L^pL^q\comma \, \(p, q\) \in \[1, \infty\] \times \[1, 2\)\comma \\ 0 \le \(A_1 - 1\)\(0\) \lesssim \jbracket{r}^{-3}\comma & \(A_1 - 1\)\(0\) \in L^q\comma \, q \in \[1, \infty\]\period\end{array}\right.\end{equation}
Combining \eqref{1_-_A_1^-sigma_calc_H2}, \eqref{A_1^-sigma_H2}, and \eqref{A_1_-_1_H2}, we obtain
\begin{equation}\label{1_-_A_1^-sigma_H2}0 \le 1 - A_1^{-\sigma} \lesssim \min\ob1, v^2\cb \lesssim \jbracket{r}^{-3}\comma \quad 1 - A_1^{-\sigma} \in L^pL^q\comma \, \(p, q, \sigma\) \in \[1, \infty\]^2 \times \N\period\end{equation}
For $\Phi_t$ we have no pointwise estimates, but only \Cref{Phi_in_Y_1}, which says
\begin{equation}\label{Phi_t_H2}\Phi_t \in L^pL^2\comma \, p \in \[1, \infty\]\period\end{equation}
It follows directly from \eqref{1_-_A_1^-sigma_H2} and \eqref{Phi_t_H2} that the third term on the RHS of \eqref{Phi_t_strichartz_estimate} can be estimated by
\begin{equation}\label{Phi_t_strichartz_estimate_third_term}\norm{\(1 - A_1^{-2}\)\Phi_t}_{L^1L^2} \lesssim 1\period\end{equation}
\subsection{Estimate for $\norm{\gradient\Phi_t\(0\)}_{L^2}$}\label{subsec:Y_2_strichartz_estimate_first_term}
Taking the gradient of \eqref{v_t_calc_H2}, we obtain
\begin{equation}\label{gradient_Phi_t_calc_H2}\gradient\Phi_t = \frac{1}{2}A_1^{-\frac{1}{2}}v_t\gradient A_1 + A_1^\frac{1}{2}\gradient v_t\period\end{equation}
By our initial conditions,
\begin{equation}\label{v_t_0_H2}\abs{v_t\(0\)} \lesssim \jbracket{r}^{-\frac{3}{2}}\comma \quad v_t\(0\) \in L^q\comma \, q \in \[2, \infty\]\period\end{equation}
Using \eqref{v_H2} and \eqref{v_0_H2}, we estimate directly
\begin{equation}\label{A_1^sigma_H2}0 \le A_1^\sigma \lesssim 1+ v^{2\sigma} \lesssim \chi_0r^{-2\sigma} + \chi_\infty\comma \quad A_1^\sigma\(0\) \lesssim 1\comma \quad A_1^\sigma\(0\) \in L^\infty\comma \, \sigma \in \[0, \infty\]\period\end{equation}
It follows from \eqref{gradient_Phi_t_calc_H2}, \eqref{A_1^-sigma_H2}, \eqref{v_t_0_H2}, and \eqref{A_1^sigma_H2} that
\begin{equation}\label{gradient_Phi_t_0_H2}\abs{\gradient\Phi_t\(0\)} \lesssim \abs{\gradient A_1\(0\)} + \abs{\gradient v_t\(0\)}\period\end{equation}
To estimate the $\gradient A_1$ term, it will require information about $\gradient v$. We relate $\gradient v$ to $\gradient\Phi$ by taking the gradient of \eqref{Phi_H0} to obtain
\begin{equation}\label{gradient_v_calc_H2}\gradient v = A_1^{-\frac{1}{2}}\(\gradient\Phi - \frac{1}{2}\int_0^vA_5^{-\frac{1}{2}}\gradient A_5 \, dy + r^{-3}\phi_{\ge \frac{1}{2}}\)\period\end{equation}
By \Cref{Phi_in_Y_1},
\begin{equation}\label{gradient_Phi_H2}\gradient\Phi \in L^pL^2\comma \, p \in \[1, \infty\]\period\end{equation}
A Maclaurin analysis of $\gradient A_5$ yields
\begin{equation}\label{gradient_A_5_H2}\abs{\gradient A_5} \lesssim \chi_0ry^4 + \chi_\infty r^{-2}\period\end{equation}
It now follows from \eqref{gradient_v_calc_H2}, \eqref{A_1^-sigma_H2}, \eqref{gradient_Phi_H2}, \eqref{A_5^-sigma_H2}, and \eqref{gradient_A_5_H2} that
\begin{equation}\label{gradient_v_H2}\abs{\gradient v} \lesssim \abs{\gradient\Phi} + \chi_0r\abs{v}^5 + \chi_\infty r^{-2}\abs{v}\period\end{equation}
Moreover, by our initial conditions,
\begin{equation}\label{gradient_v_0_H2}\abs{\gradient v\(0\)} \lesssim \jbracket{r}^{-\frac{3}{2}}\comma \quad \gradient v\(0\) \in L^q\comma \, q \in \[2, \infty\]\period\end{equation}
Now a direct differentiation of $A_1$ and a Maclaurin analysis combined with \eqref{v_H2}, \eqref{gradient_v_H2}, \eqref{v_0_H2}, and \eqref{gradient_v_0_H2} shows
\begin{equation}\label{gradient_A_1_H2}\ob\begin{array}{ll}\abs{\gradient A_1} \lesssim \abs{v}\abs{\gradient v} \lesssim \chi_0\(r^{-1}\abs{\gradient\Phi} + r^{-5}\) + \chi_\infty\(r^{-\frac{3}{2}}\abs{\gradient\Phi} + r^{-5}\)\comma & \\ \abs{\gradient A_1\(0\)} \lesssim \abs{v\(0\)}\abs{\gradient v\(0\)} \lesssim \jbracket{r}^{-3}\comma & \gradient A_1\(0\) \in L^q\comma \, q \in \[1, \infty\]\period\end{array}\right.\end{equation}
By our initial conditions,
\begin{equation}\label{gradient_v_t_0_H2}\abs{\gradient v_t\(0\)} \lesssim_\epsilon \chi_0r^{-\epsilon} + \chi_\infty r^{-\frac{3}{2}}\comma \, \epsilon > 0\comma \quad \gradient v_t\(0\) \in L^q\comma \, q \in \[2, \infty\)\period\end{equation}
It now follows from \eqref{gradient_Phi_t_0_H2}, \eqref{gradient_A_1_H2}, and \eqref{gradient_v_t_0_H2} that the first term on the RHS of \eqref{Phi_t_strichartz_estimate} can be estimated by
\begin{equation}\label{Phi_t_strichartz_estimate_first_term}\norm{\gradient\Phi_t\(0\)}_{L^2} \lesssim 1\period\end{equation}
\subsection{Estimate for $\norm{\Phi_{tt}\(0\)}_{L^2}$}\label{subsec:Y_2_strichartz_estimate_second_term}
Taking the derivative of \eqref{v_t_calc_H2} with respect to $t$, we obtain
\begin{equation}\label{Phi_tt_calc_H2}\Phi_{tt} = \frac{1}{2}A_1^{-\frac{1}{2}}\partial_tA_1v_t + A_1^\frac{1}{2}v_{tt}\period\end{equation}
Combining \eqref{v_t_calc_H2} and \eqref{A_1^-sigma_H2} leads to
\begin{equation}\label{v_t_H2}\abs{v_t} \lesssim \abs{\Phi_t}\comma \quad v_t \in L^pL^2\comma \, p \in \[1, \infty\]\period\end{equation}
Now a direct differentiation of $A_1$ combined with \eqref{v_H2}, \eqref{v_t_H2}, \eqref{v_0_H2}, and \eqref{v_t_0_H2} shows
\begin{equation}\label{dt_A_1_H2}\ob\begin{array}{ll}\abs{\partial_tA_1} \lesssim \abs{v}\abs{v_t} \lesssim \chi_0r^{-1}\abs{\Phi_t} + \chi_\infty r^{-\frac{3}{2}}\abs{\Phi_t}\comma & \partial_tA_1 \in L^pL^q\comma \, \(p, q\) \in \[1, \infty\] \times \[1, \frac{4}{3}\)\comma \\ \abs{\partial_tA_1\(0\)} \lesssim \jbracket{r}^{-3}\comma & \partial_tA_1\(0\) \in L^q\comma \, q \in \[1, \infty\]\period\end{array}\right.\end{equation}
It follows from \eqref{Phi_tt_calc_H2}, \eqref{A_1^-sigma_H2}, \eqref{dt_A_1_H2}, and \eqref{A_1^sigma_H2} that
\begin{equation}\label{Phi_tt_H2}\abs{\Phi_{tt}\(0\)} \lesssim \abs{v_t\(0\)} + \abs{v_{tt}\(0\)}\period\end{equation}
Solving \eqref{v_wave_equation} for $v_{tt}$ and evaluating at time $t = 0$, we obtain
\begin{equation}\label{v_tt_calc_H2}\abs{v_{tt}\(0\)} \le \abs{\laplacian v\(0\)} + \abs{F\(v\)\(0\)}\period\end{equation}
By our initial conditions,
\begin{equation}\label{laplacian_v_0_H2}\abs{\laplacian v\(0\)} \lesssim_\epsilon \chi_0r^{-\epsilon} + \chi_\infty r^{-\frac{3}{2}}\comma \, \epsilon > 0\comma \quad \laplacian v\(0\) \in L^q\comma \, q \in \[2, \infty\)\period\end{equation}
To estimate the $F\(v\)\(0\)$ term, we note that our initial conditions imply
\begin{equation}\label{v_r_0_H2}\abs{v_r\(0\)} \lesssim \jbracket{r}^{-\frac{3}{2}}\comma \quad v_r\(0\) \in L^q\comma \, q \in \[2, \infty\]\period\end{equation}
Using \eqref{A_1^-sigma_H2}, \Cref{F_j_analyticity}, \eqref{v_0_H2}, \eqref{v_t_0_H2}, \eqref{v_r_0_H2}, and \eqref{1_-_A_1^-sigma_H2}, we calculate eventually that
\begin{equation}\label{FFFF_H2}\abs{F\(v\)\(0\)} \lesssim \jbracket{r}^{-\frac{7}{2}}\comma \quad F\(v\)\(0\) \in L^q\comma \, q \in \(1, \infty\]\period\end{equation}
Putting \eqref{v_tt_calc_H2}, \eqref{laplacian_v_0_H2}, and \eqref{FFFF_H2} together, we obtain
\begin{equation}\label{v_tt_0_H2}\abs{v_{tt}\(0\)} \lesssim_\epsilon \chi_0r^{-\epsilon} + \chi_\infty r^{-\frac{3}{2}}\comma \, \epsilon > 0\comma \quad v_{tt}\(0\) \in L^q\comma \, q \in \[2, \infty\)\period\end{equation}
It follows from \eqref{Phi_tt_H2}, \eqref{v_t_0_H2}, and \eqref{v_tt_0_H2} that the second term on the RHS of \eqref{Phi_t_strichartz_estimate} can be estimated by
\begin{equation}\label{Phi_t_strichartz_estimate_second_term}\norm{\Phi_{tt}\(0\)}_{L^2} \lesssim 1\period\end{equation}
\subsection{Completion of the proof that $\Phi \in Y_2$}\label{subsec:Y_2_completion}
Combining \eqref{Phi_t_strichartz_estimate}, \eqref{Phi_t_strichartz_estimate_first_term}, \eqref{Phi_t_strichartz_estimate_second_term}, and \eqref{Phi_t_strichartz_estimate_third_term}, we arrive at
\begin{equation}\label{Phi_t_strichartz_estimate_result}\norm{\Phi_t}_{L^pL^\frac{4p}{p - 1}} + \norm{\gradient\Phi_t}_{Y_0} + \norm{\Phi_{tt}}_{Y_0} \lesssim 1\comma \quad p \in \[2, \infty\]\period\end{equation}
This proves $\gradient\Phi_t, \Phi_{tt} \in Y_0$, leaving us with only $\laplacian\Phi$ to analyze. Looking at \eqref{Phi_wave_equation}, we obtain (making use of \eqref{A_5^-sigma_H2})
\begin{equation}\label{laplacian_Phi_H2_calc}\abs{\laplacian\Phi} \lesssim \abs{\Phi_{tt}} + \abs{\Phi} + \abs{v} + \abs{\phi_{\ge \frac{1}{2}}r^{-3}}\period\end{equation}
By \eqref{Phi_t_strichartz_estimate_result}, \eqref{Phi_H2}, \eqref{v_H2}, and \Cref{phi_ge_1_2_properties}, this proves
\begin{equation}\label{laplacian_Phi_H2}\norm{\laplacian\Phi}_{Y_0} \lesssim 1\period\end{equation}
 By \eqref{Phi_t_strichartz_estimate_result} and \eqref{laplacian_Phi_H2}, then, we have proven $\Phi \in Y_2$.
\section{$Y_3$ Analysis of $\Phi$}\label{sec:Y_3}
In this section we upgrade the regularity of $\Phi$ to $Y_3$. In order to do this, we need to show $\gradient\laplacian\Phi, \laplacian\Phi_t, \gradient\Phi_{tt}, \Phi_{ttt} \in Y_0$. Our approach in this section is as follows. In the first subsection we introduce the Strichartz estimate that we eventually prove. In the next four subsections we derive the four estimates needed to prove the Strichartz estimate from the first subsection holds. This Strichartz estimate guarantees $\gradient\Phi_{tt}, \Phi_{ttt} \in Y_0$. Finally, in the sixth subsection, we use this information along with the wave equation \eqref{Phi_wave_equation} to prove $\laplacian\Phi_t, \gradient\laplacian\Phi \in Y_0$ and then conclude $\Phi \in Y_3$.
\subsection{Strichartz Estimate for $\Phi_{tt}$}\label{subsec:Y_3_strichartz_estimate}
 Taking the derivative of \eqref{Phi_t_wave_equation} with respect to $t$, we obtain the following wave equation for $\Phi_{tt}$:
\begin{equation}\label{Phi_tt_wave_equation}\square\Phi_{tt} = \alpha^{-2}\[2A_1^{-3}\partial_tA_1\Phi_t + \(1 - A_1^{-2}\)\Phi_{tt}\]\period\end{equation}
Therefore we have the following Strichartz estimate for $\Phi_{tt}$:
\begin{equation}\label{Phi_tt_strichartz_estimate}\begin{array}{l}\norm{\Phi_{tt}}_{L^pL^\frac{4p}{p - 1}} + \norm{\gradient\Phi_{tt}}_{Y_0} + \norm{\Phi_{ttt}}_{Y_0} \lesssim \\ \lesssim \norm{\gradient\Phi_{tt}\(0\)}_{L^2} + \norm{\Phi_{ttt}\(0\)}_{L^2} + \norm{A_1^{-3}\partial_tA_1\Phi_t}_{L^1L^2} + \norm{\(1 - A_1^{-2}\)\Phi_{tt}}_{L^1L^2}\comma \quad p \in \[2, \infty\]\period\end{array}\end{equation}
\subsection{Estimate for $\norm{A_1^{-3}\partial_tA_1\Phi_t}_{L^1L^2}$}\label{subsec:Y_3_strichartz_estimate_third_term}
We first upgrade our estimates for $\Phi$, using the fact now that $\Phi \in Y_2$, from \eqref{Phi_H2} to
\begin{equation}\label{Phi_H3}\abs{\Phi} \lesssim_\epsilon \chi_0r^{-\epsilon} + \chi_\infty r^{-\frac{3}{2}}\comma \, \epsilon > 0\comma \quad \Phi \in L^pL^q\comma \, \(p, q\) \in \[1, \infty\] \times \[2, \infty\)\period\end{equation}
Therefore we upgrade \eqref{v_H2} to
\begin{equation}\label{v_H3}\abs{v} \lesssim \abs{\Phi} + \chi_\infty r^{-3} \lesssim_\epsilon \chi_0r^{-\epsilon} + \chi_\infty r^{-\frac{3}{2}}\comma \, \epsilon > 0\comma \quad v \in L^pL^q\comma \, \(p, q\) \in \[1, \infty\] \times \[2, \infty\)\period\end{equation}
We next upgrade our estimates for $\Phi_t$, using the fact now that $\Phi_t \in Y_1$ and the Strichartz estimate \eqref{Phi_t_strichartz_estimate_result}, from \eqref{Phi_t_H2}  to
\begin{equation}\label{Phi_t_H3}\abs{\Phi_t} \lesssim \chi_0r^{-1} + \chi_\infty r^{-\frac{3}{2}}\comma \quad \Phi_t \in L^pL^q\comma \, p \in \[1, \infty\]\comma \, q \in \[2, \min\ob8, \frac{4p}{p - 1}\cb\]\period\end{equation}
Therefore we upgrade \eqref{v_t_H2} to
\begin{equation}\label{v_t_H3}\abs{v_t} \lesssim \abs{\Phi_t} \lesssim \chi_0r^{-1} + \chi_\infty r^{-\frac{3}{2}}\comma \quad v_t \in L^pL^q\comma \, p \in \[1, \infty\] \comma \, q \in \[2, \min\ob8, \frac{4p}{p - 1}\cb\]\period\end{equation}
Using \eqref{v_H3} and \eqref{v_t_H3}, we now upgrade \eqref{dt_A_1_H2} to
\begin{equation}\label{dt_A_1_H3}\ob\begin{array}{ll}\abs{\partial_tA_1} \lesssim \abs{v}\abs{v_t} \lesssim_\epsilon \chi_0r^{-1 - \epsilon} + \chi_\infty r^{-3}\comma \, \epsilon > 0\comma & \partial_tA_1 \in L^pL^q\comma \, p \in \[1, \infty\]\comma \, q \in \[1, \min\ob8, \frac{4p}{p - 1}\cb\)\comma \\ \abs{\partial_tA_1\(0\)} \lesssim \jbracket{r}^{-3}\comma & \partial_tA_1\(0\) \in L^q\comma \, q \in \[1, \infty\]\period\end{array}\right.\end{equation}
It follows directly from \eqref{A_1^-sigma_H2}, \eqref{dt_A_1_H3}, and \eqref{Phi_t_H3} that the third term on the RHS of \eqref{Phi_tt_strichartz_estimate} can be estimated by
\begin{equation}\label{Phi_tt_strichartz_estimate_third_term}\norm{A_1^{-3}\partial_tA_1\Phi_t}_{L^1L^2} \lesssim 1\period\end{equation}
\subsection{Estimate for $\norm{\(1 - A_1\)^{-2}\Phi_{tt}}_{L^1L^2}$}\label{subsec:Y_3_strichartz_estimate_fourth_term}
For $\Phi_{tt}$ we have no pointwise estimates, but $\Phi_{tt} \in Y_0$ implies
\begin{equation}\label{Phi_tt_H3}\Phi_{tt} \in L^pL^2\comma \, p \in \[1, \infty\]\period\end{equation}
It follows directly from \eqref{1_-_A_1^-sigma_H2} and \eqref{Phi_tt_H3} that the fourth term on the RHS of \eqref{Phi_tt_strichartz_estimate} can be estimated by
\begin{equation}\label{Phi_tt_strichartz_estimate_fourth_term}\norm{\(1 - A_1^{-2}\)\Phi_{tt}}_{L^1L^2} \lesssim 1\period\end{equation}
\subsection{Estimate for $\norm{\gradient\Phi_{tt}\(0\)}_{L^2}$}\label{subsec:Y_3_strichartz_estimate_first_term}
Taking the gradient of \eqref{Phi_tt_calc_H2}, we obtain
\begin{equation}\label{gradient_Phi_tt_calc_H3}\gradient\Phi_{tt} = -\frac{1}{4}A_1^{-\frac{3}{2}}\partial_tA_1v_t\gradient A_1 + \frac{1}{2}A_1^{-\frac{1}{2}}v_t\gradient\partial_tA_1 + \frac{1}{2}A_1^{-\frac{1}{2}}\partial_tA_1\gradient v_t + \frac{1}{2}A_1^{-\frac{1}{2}}v_{tt}\gradient A_1 + A_1^\frac{1}{2}\gradient v_{tt}\period\end{equation}
We next upgrade our estimates for $\gradient\Phi$, using the fact now that $\gradient\Phi \in Y_1$, from \eqref{gradient_Phi_H2} to
\begin{equation}\label{gradient_Phi_H3}\abs{\gradient\Phi} \lesssim \chi_0r^{-1} + \chi_\infty r^{-\frac{3}{2}}\comma \quad \gradient\Phi \in L^pL^q\comma \, \(p, q\) \in \[1, \infty\] \times \[2, 4\)\period\end{equation}
Therefore we upgrade \eqref{gradient_v_H2} to
\begin{equation}\label{gradient_v_H3}\abs{\gradient v} \lesssim \abs{\gradient\Phi} + \chi_0r\abs{v}^5 + \chi_\infty r^{-2}\abs{v} \lesssim \chi_0r^{-1} + \chi_\infty r^{-\frac{3}{2}}\comma \quad \gradient v \in L^pL^q\comma \, \(p, q\) \in \[1, \infty\] \times \[2, 4\)\period\end{equation}
Using \eqref{v_H3} and \eqref{gradient_v_H3}, we now upgrade \eqref{gradient_A_1_H2} to
\begin{equation}\label{gradient_A_1_H3}\ob\begin{array}{ll}\abs{\gradient A_1} \lesssim \abs{v}\abs{\gradient v} \lesssim_\epsilon \chi_0r^{-1 - \epsilon} + \chi_\infty r^{-3}\comma \, \epsilon > 0\comma & \gradient A_1 \in L^pL^q\comma \, \(p, q\) \in \[1, \infty\] \times \[1, 4\)\comma \\ \abs{\gradient A_1\(0\)} \lesssim \abs{v\(0\)}\abs{\gradient v\(0\)} \lesssim \jbracket{r}^{-3}\comma & \gradient A_1\(0\) \in L^q\comma \, q \in \[1, \infty\]\period\end{array}\right.\end{equation}
Using \eqref{v_H3}, we also upgrade \eqref{A_1^sigma_H2} to
\begin{equation}\label{A_1^sigma_H3}0 \le A_1^\sigma \lesssim 1+ v^{2\sigma} \lesssim \chi_0r^{-\epsilon} + \chi_\infty\comma \, \epsilon > 0\comma \quad A_1^\sigma\(0\) \lesssim 1\comma \quad A_1^\sigma\(0\) \in L^\infty\comma \, \sigma \in \[0, \infty\]\period\end{equation}
It follows from \eqref{gradient_Phi_tt_calc_H3}, \eqref{A_1^-sigma_H2}, \eqref{dt_A_1_H3}, \eqref{v_t_0_H2}, \eqref{gradient_A_1_H3}, and \eqref{A_1^sigma_H3} that
\begin{equation}\label{gradient_Phi_tt_H3}\abs{\gradient\Phi_{tt}\(0\)} \lesssim \abs{\gradient A_1\(0\)} + \abs{\gradient\partial_tA_1\(0\)} + \abs{\gradient v_t\(0\)} + \abs{v_{tt}\(0\)} + \abs{\gradient v_{tt}\(0\)}\period\end{equation}
The first, third, and fourth terms on the RHS of \eqref{gradient_Phi_tt_H3} are already under control. For the second term, two direct differentiations of $A_1$, a Maclaurin analysis, \eqref{v_0_H2}, \eqref{gradient_v_t_0_H2}, \eqref{v_t_0_H2}, and \eqref{gradient_v_0_H2} show
\begin{equation}\label{gradient_dt_A_1_H3}\abs{\gradient\partial_tA_1\(0\)} \lesssim \abs{v\(0\)}\abs{\gradient v_t\(0\)} + \abs{v_t\(0\)}\abs{\gradient v\(0\)} \lesssim_\epsilon \chi_0r^{-\epsilon} + \chi_\infty r^{-3}\comma \, \epsilon > 0\comma \quad \gradient\partial_tA_1\(0\) \in L^q\comma \, q \in \[1, \infty\)\period\end{equation}
For the fifth term on the RHS of \eqref{gradient_Phi_tt_H3}, we take the gradient of \eqref{v_wave_equation} and solve for $\gradient v_{tt}$ to obtain
\begin{equation}\label{gradient_v_tt_calc_H3}\abs{\gradient v_{tt}\(0\)} \le \abs{\gradient\laplacian v\(0\)} + \abs{\gradient\[F\(v\)\]\(0\)}\period\end{equation}
By our initial conditions,
\begin{equation}\label{gradient_laplacian_v_0_H3}\abs{\gradient\laplacian v\(0\)} \lesssim \chi_0r^{-1} + \chi_\infty r^{-\frac{3}{2}}\comma \quad \gradient\laplacian v\(0\) \in L^q \comma \, q \in \[2, 4\)\period\end{equation}
To estimate the $\gradient\[F\(v\)\]\(0\)$ term, we note that our initial conditions imply
\begin{equation}\label{gradient_v_r_0_H3}\abs{\gradient v_r\(0\)} \lesssim_\epsilon \chi_0r^{-\epsilon} + \chi_\infty r^{-\frac{3}{2}}\comma \, \epsilon > 0\comma \quad \gradient v_r\(0\) \in L^q\comma \, q \in \[2, \infty\)\period\end{equation}
Using \eqref{A_1^-sigma_H2}, \eqref{gradient_A_1_H3}, \Cref{F_j_analyticity}, \eqref{v_0_H2}, \eqref{gradient_v_0_H2}, \eqref{v_t_0_H2}, \eqref{v_r_0_H2}, \eqref{gradient_v_t_0_H2}, \eqref{gradient_v_r_0_H3}, and \eqref{1_-_A_1^-sigma_H2}, we calculate eventually that
\begin{equation}\label{gradient_FFFF_H3}\abs{\gradient\[F\(v\)\]\(0\)} \lesssim_\epsilon \chi_0r^{-\epsilon} + \chi_\infty r^{-\frac{7}{2}}\comma \, \epsilon > 0\comma \quad \gradient\[F\(v\)\]\(0\) \in L^q\comma \, q \in \(1, \infty\)\period\end{equation}
Putting \eqref{gradient_v_tt_calc_H3}, \eqref{gradient_laplacian_v_0_H3}, and \eqref{gradient_FFFF_H3} together, we obtain
\begin{equation}\label{gradient_v_tt_0_H3}\abs{\gradient v_{tt}\(0\)} \lesssim \chi_0r^{-1} + \chi_\infty r^{-\frac{3}{2}}\comma \quad \gradient v_{tt}\(0\) \in L^q\comma \, q \in \[2, 4\)\period\end{equation}
It now follows from \eqref{gradient_Phi_tt_H3}, \eqref{gradient_A_1_H3}, \eqref{gradient_dt_A_1_H3}, \eqref{gradient_v_t_0_H2}, \eqref{v_tt_0_H2}, and \eqref{gradient_v_tt_0_H3} that the first term on the RHS of \eqref{Phi_tt_strichartz_estimate} can be estimated by
\begin{equation}\label{Phi_tt_strichartz_estimate_first_term}\norm{\gradient\Phi_{tt}\(0\)}_{L^2} \lesssim 1\period\end{equation}
\subsection{Estimate for $\norm{\Phi_{ttt}\(0\)}_{L^2}$}\label{subsec:Y_3_strichartz_estimate_second_term}
Taking the derivative of \eqref{Phi_tt_calc_H2} with respect to $t$, we obtain
\begin{equation}\label{Phi_ttt_calc_H3}\Phi_{ttt} = -\frac{1}{4}A_1^{-\frac{3}{2}}\(\partial_tA_1\)^2v_t + \frac{1}{2}A_1^{-\frac{1}{2}}\partial_{tt}A_1v_t + A_1^{-\frac{1}{2}}\partial_tA_1v_{tt} + A_1^\frac{1}{2}v_{ttt}\period\end{equation}
It follows from \eqref{Phi_ttt_calc_H3}, \eqref{A_1^-sigma_H2}, \eqref{dt_A_1_H3}, and \eqref{A_1^sigma_H3} that
\begin{equation}\label{Phi_ttt_H3}\abs{\Phi_{ttt}\(0\)} \lesssim \abs{v_t\(0\)} + \abs{\partial_{tt}A_1\(0\)} + \abs{v_{tt}\(0\)} + \abs{v_{ttt}\(0\)}\period\end{equation}
The first and third terms on the RHS we already have under control. For the second term, we first relate $v_{tt}$ to $\Phi_{tt}$ by differentiating \eqref{v_t_calc_H2} with respect to $t$ to obtain
\begin{equation}\label{v_tt_calc_H3}v_{tt} = -\frac{1}{2}A_1^{-\frac{3}{2}}\partial_tA_1\Phi_t + A_1^{-\frac{1}{2}}\Phi_{tt}\period\end{equation}
Combining \eqref{v_tt_calc_H3}, \eqref{A_1^-sigma_H2}, \eqref{dt_A_1_H3}, \eqref{Phi_t_H3}, \eqref{Phi_tt_H3}, \eqref{v_H3}, and \eqref{v_t_H3} leads to
\begin{equation}\label{v_tt_H3}\abs{v_{tt}} \lesssim \abs{v}\abs{v_t}\abs{\Phi_t} + \abs{\Phi_{tt}} \lesssim_\epsilon \chi_0\(r^{-2 - \epsilon} + \abs{\Phi_{tt}}\) + \chi_\infty\(r^{-\frac{9}{2}} + \abs{\Phi_{tt}}\)\comma \, \epsilon > 0\comma \quad v_{tt} \in L^pL^2\comma \, p \in \[1, \infty\]\period\end{equation}
Now, two direct differentiations of $A_1$, \eqref{v_H3}, \eqref{v_tt_H3}, \eqref{v_t_H3}, \eqref{v_0_H2}, \eqref{v_tt_0_H2}, and \eqref{v_t_0_H2} show
\begin{equation}\label{dtt_A_1_H3}\ob\begin{array}{ll}\abs{\partial_{tt}A_1} \lesssim \abs{v}\abs{v_{tt}} + v_t^2 \lesssim_\epsilon \chi_0\(r^{-2 - \epsilon} + \abs{\Phi_{tt}}\) + \chi_\infty\(r^{-3} + r^{-\frac{3}{2}}\abs{\Phi_{tt}}\)\comma \, \epsilon > 0\comma & \partial_{tt}A_1 \in L^pL^q\comma \, \(p, q\) \in \[1, \infty\] \times \[1, 2\)\comma \\ \abs{\partial_{tt}A_1\(0\)} \lesssim_\epsilon \chi_0r^{-\epsilon} + \chi_\infty r^{-3}\comma \, \epsilon > 0\comma & \partial_{tt}A_1\(0\) \in L^q\comma \, q \in \[1, \infty\)\period\end{array}\right.\end{equation}
For the fourth term on the RHS of \eqref{Phi_ttt_H3}, we differentiate \eqref{v_wave_equation} with respect to $t$ and solve for $v_{ttt}$ to obtain
\begin{equation}\label{v_ttt_calc_H3}\abs{v_{ttt}\(0\)} \le \abs{\laplacian v_t\(0\)} + \abs{\partial_t\[F\(v\)\]\(0\)}\period\end{equation}
By our initial conditions,
\begin{equation}\label{laplacian_v_t_0_H3}\abs{\laplacian v_t\(0\)} \lesssim \chi_0r^{-1} + \chi_\infty r^{-\frac{3}{2}}\comma \quad \laplacian v_t\(0\) \in L^q\comma \, q \in \[2, 4\)\period\end{equation}
To estimate the $\partial_t\[F\(v\)\]\(0\)$ term, we note that our initial conditions imply
\begin{equation}\label{v_rt_0_H3}\abs{v_{rt}\(0\)} \lesssim_\epsilon \chi_0r^{-\epsilon} + \chi_\infty r^{-\frac{3}{2}}\comma \, \epsilon > 0\comma \quad v_{rt}\(0\) \in L^q\comma \, q \in \[2, \infty\)\period\end{equation}
Using \eqref{A_1^-sigma_H2}, \eqref{dt_A_1_H3}, \Cref{F_j_analyticity}, \eqref{v_0_H2}, \eqref{v_t_0_H2}, \eqref{v_r_0_H2}, \eqref{v_tt_0_H2}, \eqref{v_rt_0_H3}, and \eqref{1_-_A_1^-sigma_H2}, we calculate eventually that
\begin{equation}\label{FFFF_t_H3}\abs{\partial_t\[F\(v\)\]\(0\)} \lesssim_\epsilon \chi_0r^{-\epsilon}  + \chi_\infty r^{-\frac{7}{2}}\comma \, \epsilon > 0\comma \quad \partial_t\[F\(v\)\]\(0\) \in L^q\comma \, q \in \(1, \infty\)\period\end{equation}
Putting \eqref{v_ttt_calc_H3}, \eqref{laplacian_v_t_0_H3}, and \eqref{FFFF_t_H3} together, we obtain
\begin{equation}\label{v_ttt_0_H3}\abs{v_{ttt}\(0\)} \lesssim \chi_0r^{-1} + \chi_\infty r^{-\frac{3}{2}}\comma \quad v_{ttt}\(0\) \in L^q\comma \, q \in \[2, 4\)\period\end{equation}
It now follows from \eqref{Phi_ttt_H3}, \eqref{v_t_0_H2}, \eqref{dtt_A_1_H3}, \eqref{v_tt_0_H2}, and \eqref{v_ttt_0_H3} that the second term on the RHS of \eqref{Phi_tt_strichartz_estimate} can be estimated by
\begin{equation}\label{Phi_tt_strichartz_estimate_second_term}\norm{\Phi_{ttt}\(0\)}_{L^2} \lesssim 1\period\end{equation}
\subsection{Completion of the Proof that $\Phi \in Y_3$}\label{subsec:Y_3_completion}
Combining \eqref{Phi_tt_strichartz_estimate}, \eqref{Phi_tt_strichartz_estimate_first_term}, \eqref{Phi_tt_strichartz_estimate_second_term}, \eqref{Phi_tt_strichartz_estimate_third_term}, and \eqref{Phi_tt_strichartz_estimate_fourth_term}, we arrive at
\begin{equation}\label{Phi_tt_strichartz_estimate_result}\norm{\Phi_{tt}}_{L^pL^\frac{4p}{p - 1}} + \norm{\gradient\Phi_{tt}}_{Y_0} + \norm{\Phi_{ttt}}_{Y_0} \lesssim 1\comma \quad p \in \[2, \infty\]\period\end{equation}
This proves $\gradient\Phi_{tt}, \Phi_{ttt} \in Y_0$, leaving us with only $\gradient\laplacian\Phi$ and $\laplacian\Phi_t$ to analyze. Looking at \eqref{Phi_t_wave_equation}, we obtain
\begin{equation}\label{laplacian_Phi_t_H3_calc}\abs{\laplacian\Phi_t} \lesssim \abs{\Phi_{ttt}} + \abs{\(1 - A_1^{-2}\)\Phi_t}\period\end{equation}
By \eqref{Phi_tt_strichartz_estimate_result}, \eqref{1_-_A_1^-sigma_H2}, and \eqref{Phi_t_H3} this proves
\begin{equation}\label{laplacian_Phi_t_H3}\norm{\laplacian\Phi_t}_{Y_0} \lesssim 1\period\end{equation}
This leaves us with only $\gradient\laplacian\Phi$ to analyze. Going back to the wave equation \eqref{Phi_wave_equation} for $\Phi$, taking its gradient, and solving for $\gradient\laplacian\Phi$, we obtain
\begin{equation}\label{gradient_laplacian_Phi_calc_H3}\abs{\gradient\laplacian\Phi} \lesssim \abs{\gradient\Phi_{tt}} + \abs{\gradient\Phi} + \gradient\(\int_0^vA_5^{-\frac{3}{2}} \, dy\) + \gradient\(\phi_{\ge \frac{1}{2}}r^{-3}\)\period\end{equation}
Making use of \eqref{gradient_laplacian_Phi_calc_H3}, \eqref{Phi_tt_strichartz_estimate_result}, \eqref{gradient_Phi_H3}, and \Cref{phi_ge_1_2_properties}, we conclude
\begin{equation}\label{gradient_laplacian_Phi_calc_2_H3}\norm{\gradient\laplacian\Phi}_{Y_0} \lesssim 1 + \norm{\gradient\(\int_0^vA_5^{-\frac{3}{2}} \, dy\)}_{Y_0}\period\end{equation}
We now calculate directly that
\begin{equation}\label{gradient_int_A_5_calc_H3}\gradient\(\int_0^vA_5^{-\frac{3}{2}} \, dy\) = A_1^{-\frac{3}{2}}\gradient v - \frac{3}{2}\int_0^vA_5^{-\frac{5}{2}}\gradient A_5 \, dy\period\end{equation}
Therefore, reusing the Maclaurin analysis that was done in order to derive \eqref{gradient_v_H2} and using \eqref{gradient_v_H3} and \eqref{v_H3}, we estimate
\begin{equation}\label{gradient_int_A_5_H3}\abs{\gradient\(\int_0^vA_5^{-\frac{3}{2}} \, dy\)} \lesssim \abs{\gradient v} + \chi_0r\abs{v}^5 + \chi_\infty r^{-2}\abs{v} \lesssim \chi_0r^{-1} + \chi_\infty r^{-\frac{3}{2}}\period\end{equation}
Estimate \eqref{gradient_int_A_5_H3} certainly implies that \eqref{gradient_laplacian_Phi_calc_2_H3} can be upgraded to
\begin{equation}\label{gradient_laplacian_Phi_H3}\norm{\gradient\laplacian\Phi}_{Y_0} \lesssim 1\period\end{equation}
By \eqref{Phi_tt_strichartz_estimate_result}, \eqref{laplacian_Phi_t_H3}, and \eqref{gradient_laplacian_Phi_H3}, then, we have proven $\Phi \in Y_3$.
\section{$Y_4$ Analysis of $\Phi$}\label{sec:Y_4}
In this section we upgrade the regularity of $\Phi$ to $Y_4$. In order to do this, we need to show $\laplacian^2\Phi, \gradient\laplacian\Phi_t, \laplacian\Phi_{tt}, \gradient\Phi_{ttt}, \Phi_{tttt} \in Y_0$. Our approach in this section is as follows. In the first subsection we introduce the Strichartz estimate that we eventually prove. In the next six subsections we derive the six estimates needed to prove the Strichartz estimate from the first subsection holds. This Strichartz estimate guarantees $\gradient\Phi_{ttt}, \Phi_{tttt} \in Y_0$. Finally, in the eighth subsection, we use this information along with the wave equation \eqref{Phi_wave_equation} to prove $\laplacian\Phi_{tt}, \gradient\laplacian\Phi_t, \laplacian^2\Phi \in Y_0$ and then conclude $\Phi \in Y_4$.
\subsection{Strichartz Estimate on $\square\Phi_{ttt}$}\label{subsec:Y_4_strichartz_estimate}
Taking the derivative of \eqref{Phi_tt_wave_equation} with respect to $t$, we obtain the following wave equation for $\Phi_{ttt}$:
\begin{equation}\label{Phi_ttt_wave_equation}\square\Phi_{ttt} = \alpha^{-2}\[-6A_1^{-4}\(\partial_tA_1\)^2\Phi_t + 2A_1^{-3}\partial_{tt}A_1\Phi_t + 4A_1^{-3}\partial_tA_1\Phi_{tt} + \(1 - A_1^{-2}\)\Phi_{ttt}\]\period\end{equation}
Therefore we have the following Strichartz estimate for $\Phi_{ttt}$:
\begin{equation}\label{Phi_ttt_strichartz_estimate}\begin{array}{l}\norm{\Phi_{ttt}}_{L^pL^\frac{4p}{p - 1}} + \norm{\gradient\Phi_{ttt}}_{Y_0} + \norm{\Phi_{tttt}}_{Y_0} \lesssim \\ \lesssim \norm{\gradient\Phi_{ttt}\(0\)}_{L^2} + \norm{\Phi_{tttt}\(0\)}_{L^2} + \norm{A_1^{-4}\(\partial_tA_1\)^2\Phi_t}_{L^1L^2} + \\ + \norm{A_1^{-3}\partial_{tt}A_1\Phi_t}_{L^1L^2} + \norm{A_1^{-3}\partial_tA_1\Phi_{tt}}_{L^1L^2} + \norm{\(1 - A_1^{-2}\)\Phi_{ttt}}_{L^1L^2}\comma \, p \in \[2, \infty\]\period\end{array}\end{equation}
\subsection{Estimate for $\norm{A_1^{-4}\(\partial_tA_1\)^2\Phi_t}_{L^1L^2}$}\label{subsec:Y_4_strichartz_estimate_third_term}
We first upgrade our estimates for $\Phi$, using the fact now that $\Phi \in Y_3$, from \eqref{Phi_H3} to
\begin{equation}\label{Phi_H4}\abs{\Phi} \lesssim \jbracket{r}^{-\frac{3}{2}}\comma \quad \Phi \in L^pL^q\comma \, \(p, q\) \in \[1, \infty\] \times \[2, \infty\]\period\end{equation}
Therefore we upgrade \eqref{v_H3} to
\begin{equation}\label{v_H4}\abs{v} \lesssim \abs{\Phi} + \chi_\infty r^{-3} \lesssim \jbracket{r}^{-\frac{3}{2}}\comma \quad v \in L^pL^q\comma \, \(p, q\) \in \[1, \infty\] \times \[2, \infty\]\period\end{equation}
We next upgrade our estimates for $\Phi_t$, using the fact now that $\Phi_t \in Y_2$, from \eqref{Phi_t_H3} to
\begin{equation}\label{Phi_t_H4}\abs{\Phi_t} \lesssim_\epsilon \chi_0r^{-\epsilon} + \chi_\infty r^{-\frac{3}{2}}\comma \, \epsilon > 0\comma \quad \Phi_t \in L^pL^q\comma \, \(p, q\) \in \[1, \infty\] \times \[2, \infty\)\period\end{equation}
Therefore we upgrade \eqref{v_t_H3} to
\begin{equation}\label{v_t_H4}\abs{v_t} \lesssim \abs{\Phi_t} \lesssim_\epsilon \chi_0r^{-\epsilon} + \chi_\infty r^{-\frac{3}{2}}\comma \, \epsilon > 0\comma \quad v_t \in L^pL^q\comma \, \(p, q\) \in \[1, \infty\] \times \[2, \infty\)\period\end{equation}
Using \eqref{v_H4} and \eqref{v_t_H4}, we now upgrade \eqref{dt_A_1_H3} to
\begin{equation}\label{dt_A_1_H4}\ob\begin{array}{ll}\abs{\partial_tA_1} \lesssim \abs{v}\abs{v_t} \lesssim_\epsilon \chi_0r^{-\epsilon} + \chi_\infty r^{-3}\comma \, \epsilon > 0\comma & \partial_tA_1 \in L^pL^q\comma \, \(p, q\) \in \[1, \infty\] \times \[1, \infty\)\comma \\ \abs{\partial_tA_1\(0\)} \lesssim \jbracket{r}^{-3}\comma & \partial_tA_1\(0\) \in L^q\comma \, q \in \[1, \infty\]\period\end{array}\right.\end{equation}
It follows directly from \eqref{A_1^-sigma_H2}, \eqref{dt_A_1_H4}, and \eqref{Phi_t_H4} that the third term on the RHS of \eqref{Phi_ttt_strichartz_estimate} can be estimated by
\begin{equation}\label{Phi_ttt_strichartz_estimate_third_term}\norm{A_1^{-4}\(\partial_tA_1\)^2\Phi_t}_{L^1L^2} \lesssim 1\period\end{equation}
\subsection{Estimate for $\norm{A_1^{-3}\partial_{tt}A_1\Phi_t}_{L^1L^2}$}\label{subsec:Y_4_strichartz_estimate_fourth_term}
We next upgrade our estimates for $\Phi_{tt}$, using the fact now that $\Phi_{tt} \in Y_1$ and the Strichartz estimate \eqref{Phi_tt_strichartz_estimate_result}, from \eqref{Phi_tt_H3} to
\begin{equation}\label{Phi_tt_H4}\abs{\Phi_{tt}} \lesssim \chi_0r^{-1} + \chi_\infty r^{-\frac{3}{2}}\comma \quad \Phi_{tt} \in L^pL^q\comma \, p \in \[1, \infty\]\comma \, q \in \[2, \min\ob8, \frac{4p}{p - 1}\cb\]\period\end{equation}
Therefore we upgrade \eqref{v_tt_H3} to
\begin{equation}\label{v_tt_H4}\abs{v_{tt}} \lesssim \abs{v}\abs{v_t}\abs{\Phi_t} + \abs{\Phi_{tt}} \lesssim \chi_0r^{-1} + \chi_\infty r^{-\frac{3}{2}}\comma \quad v_{tt} \in L^pL^q\comma \, p \in \[1, \infty\]\comma \, q \in \[2, \min\ob8, \frac{4p}{p - 1}\cb\]\period\end{equation}
Using \eqref{v_H4}, \eqref{v_tt_H4}, and \eqref{v_t_H4}, we upgrade \eqref{dtt_A_1_H3} to
\begin{equation}\label{dtt_A_1_H4}\ob\begin{array}{ll}\abs{\partial_{tt}A_1} \lesssim \abs{v}\abs{v_{tt}} + v_t^2 \lesssim_\epsilon \chi_0r^{-1} + \chi_\infty r^{-3}\comma & \partial_{tt}A_1 \in L^pL^q\comma \, p \in \[1, \infty\] \comma \, q \in \[1, \min\ob8, \frac{4p}{p - 1}\cb\]\comma \\ \abs{\partial_{tt}A_1\(0\)} \lesssim_\epsilon \chi_0r^{-\epsilon} + \chi_\infty r^{-3}\comma \, \epsilon > 0\comma & \partial_{tt}A_1\(0\) \in L^q\comma \, q \in \[1, \infty\)\period\end{array}\right.\end{equation}
It follows directly from \eqref{A_1^-sigma_H2}, \eqref{dtt_A_1_H4}, and \eqref{Phi_t_H4} that the fourth term on the RHS of \eqref{Phi_ttt_strichartz_estimate} can be estimated by
\begin{equation}\label{Phi_ttt_strichartz_estimate_fourth_term}\norm{A_1^{-3}\partial_{tt}A_1\Phi_t}_{L^1L^2} \lesssim 1\period\end{equation}
\subsection{Estimate for $\norm{A_1^{-3}\partial_tA_1\Phi_{tt}}_{L^1L^2}$}\label{subsec:Y_4_strichartz_estimate_fifth_term}
It follows directly from \eqref{A_1^-sigma_H2}, \eqref{dt_A_1_H4}, and \eqref{Phi_tt_H4} that the fifth term on the RHS of \eqref{Phi_ttt_strichartz_estimate} can be estimated by
\begin{equation}\label{Phi_ttt_strichartz_estimate_fifth_term}\norm{A_1^{-3}\partial_tA_1\Phi_{tt}}_{L^1L^2} \lesssim 1\period\end{equation}
\subsection{Estimate for $\norm{\(1 - A_1\)^{-2}\Phi_{ttt}}_{L^1L^2}$}\label{subsec:Y_4_strichartz_estimate_sixth_term}
For $\Phi_{ttt}$ we have no pointwise estimates, but $\Phi_{ttt} \in Y_0$ implies
\begin{equation}\label{Phi_ttt_H4}\Phi_{ttt} \in L^pL^2\comma \, p \in \[1, \infty\]\period\end{equation}
It follows directly from \eqref{1_-_A_1^-sigma_H2} and \eqref{Phi_ttt_H4} that the sixth term on the RHS of \eqref{Phi_ttt_strichartz_estimate} can be estimated by
\begin{equation}\label{Phi_ttt_strichartz_estimate_sixth_term}\norm{\(1 - A_1^{-2}\)\Phi_{ttt}}_{L^1L^2} \lesssim 1\period\end{equation}
\subsection{Estimate for $\norm{\gradient\Phi_{ttt}\(0\)}_{L^2}$}\label{subsec:Y_4_strichartz_estimate_first_term}
Taking the gradient of \eqref{Phi_ttt_calc_H3}, we obtain
\begin{equation}\label{gradient_Phi_ttt_calc_H4}\gradient\Phi_{ttt} = \ob\begin{array}{l}\frac{3}{8}A_1^{-\frac{5}{2}}\(\partial_tA_1\)^2v_t\gradient A_1 - \frac{1}{2}A_1^{-\frac{3}{2}}\partial_tA_1v_t\gradient\partial_tA_1 - \frac{1}{4}A_1^{-\frac{3}{2}}\(\partial_tA_1\)^2\gradient v_t - \\ - \frac{1}{4}A_1^{-\frac{3}{2}}\partial_{tt}A_1v_t\gradient A_1 + \frac{1}{2}A_1^{-\frac{1}{2}}v_t\gradient\partial_{tt}A_1 + \frac{1}{2}A_1^{-\frac{1}{2}}\partial_{tt}A_1\gradient v_t - \frac{1}{2}A_1^{-\frac{3}{2}}\partial_tA_1v_{tt}\gradient A_1 + \\ + A_1^{-\frac{1}{2}}v_{tt}\gradient\partial_tA_1 + A_1^{-\frac{1}{2}}\partial_tA_1\gradient v_{tt} + \frac{1}{2}A_1^{-\frac{1}{2}}v_{ttt}\gradient A_1 + A_1^\frac{1}{2}\gradient v_{ttt}\period\end{array}\right.\end{equation}
We next upgrade our estimates for $\gradient\Phi$, using the fact now that $\gradient\Phi \in Y_2$, from \eqref{gradient_Phi_H3} to
\begin{equation}\label{gradient_Phi_H4}\abs{\gradient\Phi} \lesssim_\epsilon \chi_0r^{-\epsilon} + \chi_\infty r^{-\frac{3}{2}}\comma \, \epsilon > 0\comma \quad \gradient\Phi \in L^pL^q\comma \, \(p, q\) \in \[1, \infty\] \times \[2, \infty\)\period\end{equation}
Therefore we upgrade \eqref{gradient_v_H3} to
\begin{equation}\label{gradient_v_H4}\abs{\gradient v} \lesssim \abs{\gradient\Phi} + r\abs{v}^5 + \chi_\infty r^{-2}\abs{v} \lesssim_\epsilon \chi_0r^{-\epsilon} + \chi_\infty r^{-\frac{3}{2}}\comma \, \epsilon > 0\comma \quad \gradient v \in L^pL^q\comma \, \(p, q\) \in \[1, \infty\] \times \[2, \infty\)\period\end{equation}
Using \eqref{v_H4} and \eqref{gradient_v_H4}, we now upgrade \eqref{gradient_A_1_H3} to
\begin{equation}\label{gradient_A_1_H4}\ob\begin{array}{ll}\abs{\gradient A_1} \lesssim \abs{v}\abs{\gradient v} \lesssim_\epsilon \chi_0r^{-\epsilon} + \chi_\infty r^{-3}\comma \, \epsilon > 0\comma & \gradient A_1 \in L^pL^q\comma \, \(p, q\) \in \[1, \infty\] \times \[1, \infty\)\comma \\ \abs{\gradient A_1\(0\)} \lesssim \abs{v\(0\)}\abs{\gradient v\(0\)} \lesssim \jbracket{r}^{-3}\comma & \gradient A_1\(0\) \in L^q\comma \, q \in \[1, \infty\]\period\end{array}\right.\end{equation}
Using \eqref{v_H4}, we also upgrade \eqref{A_1^sigma_H3} to
\begin{equation}\label{A_1^sigma_H4}0 \le A_1^\sigma \lesssim 1+ v^{2\sigma} \lesssim 1\comma \quad A_1^\sigma \in L^pL^\infty\comma \, \(p, \sigma\) \in \[1, \infty\] \times \[0, \infty\]\period\end{equation}
It follows from \eqref{gradient_Phi_ttt_calc_H4}, \eqref{A_1^-sigma_H2}, \eqref{dt_A_1_H4}, \eqref{v_t_0_H2}, \eqref{gradient_A_1_H4}, and \eqref{A_1^sigma_H4} that
\begin{equation}\label{gradient_Phi_ttt_H4}\abs{\gradient\Phi_{ttt}\(0\)} \lesssim \ob\begin{array}{l}\abs{\gradient A_1\(0\)} + \abs{\gradient\partial_tA_1\(0\)} + \abs{\gradient v_t\(0\)} + \abs{\partial_{tt}A_1\(0\)} + \abs{\gradient\partial_{tt}A_1\(0\)} + \\ + \abs{v_{tt}\(0\)} + \abs{\gradient v_{tt}\(0\)} + \abs{v_{ttt}\(0\)} + \abs{\gradient v_{ttt}\(0\)}\period\end{array}\right.\end{equation}
All of the terms on the RHS except the fifth and ninth we already have under control. For the fifth term, three direct differentiations of $A_1$, a Maclaurin analysis, \eqref{v_t_0_H2}, \eqref{gradient_v_t_0_H2}, \eqref{gradient_v_0_H2}, \eqref{v_tt_0_H2}, \eqref{v_0_H2}, and \eqref{gradient_v_tt_0_H3} show
\begin{equation}\label{gradient_dtt_A_1_H4}\ob\begin{array}{l}\abs{\gradient\partial_{tt}A_1\(0\)} \lesssim \abs{v_t\(0\)}\abs{\gradient v_t\(0\)} + \abs{\gradient v\(0\)}\abs{v_{tt}\(0\)} + \abs{v\(0\)}\abs{\gradient v_{tt}\(0\)} \lesssim \chi_0r^{-1} + \chi_\infty r^{-3}\comma \\ \gradient\partial_{tt}A_1\(0\) \in L^q\comma \, q \in \[1, 4\)\period\end{array}\right.\end{equation}
For the ninth term on the RHS of \eqref{gradient_Phi_ttt_H4}, we differentiate \eqref{v_wave_equation} with respect to $t$, take the gradient, and solve for $\gradient v_{ttt}$ to obtain
\begin{equation}\label{gradient_v_ttt_calc_H4}\abs{\gradient v_{ttt}\(0\)} \le \abs{\gradient\laplacian v_t\(0\)} + \abs{\gradient\partial_t\[F\(v\)\]\(0\)}\period\end{equation}
By our initial conditions,
\begin{equation}\label{gradient_laplacian_v_t_0_H4}\gradient\laplacian v_t\(0\) \in L^2\period\end{equation}
To estimate the $\gradient\partial_t\[F\(v\)\]\(0\)$ term, we note that our initial conditions imply
\begin{equation}\label{gradient_v_rt_0_H4}\abs{\gradient v_{rt}\(0\)} \lesssim \chi_0r^{-1} + \chi_\infty r^{-\frac{3}{2}}\comma \quad \gradient v_{rt}\(0\) \in L^q\comma \, q \in \[2, 4\)\period\end{equation}
Using \eqref{A_1^-sigma_H2}, \eqref{dt_A_1_H4}, \eqref{gradient_A_1_H4}, \eqref{gradient_dt_A_1_H3}, \Cref{F_j_analyticity}, \eqref{v_0_H2}, \eqref{v_t_0_H2}, \eqref{gradient_v_0_H2}, \eqref{gradient_v_t_0_H2}, \eqref{v_r_0_H2}, \eqref{v_tt_0_H2}, \eqref{v_rt_0_H3}, \eqref{gradient_v_r_0_H3}, \eqref{gradient_v_tt_0_H3}, \eqref{gradient_v_rt_0_H4}, and \eqref{1_-_A_1^-sigma_H2}, we calculate eventually that
\begin{equation}\label{gradient_FFFF_t_H4}\abs{\gradient\partial_t\[F\(v\)\]\(0\)} \lesssim \chi_0r^{-1} + \chi_\infty r^{-\frac{7}{2}}\comma \quad \gradient\partial_t\[F\(v\)\]\(0\) \in L^q\comma \, q \in \(1, 4\)\period\end{equation}
Putting \eqref{gradient_v_ttt_calc_H4}, \eqref{gradient_laplacian_v_t_0_H4}, and \eqref{gradient_FFFF_t_H4} together, we obtain
\begin{equation}\label{gradient_v_ttt_0_H4}\gradient v_{ttt}\(0\) \in L^2\period\end{equation}
It now follows from \eqref{gradient_Phi_ttt_H4}, \eqref{gradient_A_1_H4}, \eqref{gradient_dt_A_1_H3}, \eqref{gradient_v_t_0_H2}, \eqref{dtt_A_1_H4}, \eqref{gradient_dtt_A_1_H4}, \eqref{v_tt_0_H2}, \eqref{gradient_v_tt_0_H3}, \eqref{v_ttt_0_H3}, and \eqref{gradient_v_ttt_0_H4} that the first term on the RHS of \eqref{Phi_ttt_strichartz_estimate} can be estimated by
\begin{equation}\label{Phi_ttt_strichartz_estimate_first_term}\norm{\gradient\Phi_{ttt}\(0\)}_{L^2} \lesssim 1\period\end{equation}
\subsection{Estimate for $\norm{\Phi_{tttt}\(0\)}_{L^2}$}\label{subsec:Y_4_strichartz_estimate_second_term}
Taking the derivative of \eqref{Phi_ttt_calc_H3} with respect to $t$, we obtain
\begin{equation}\label{Phi_tttt_calc_H4}\Phi_{tttt} = \ob\begin{array}{l}\frac{3}{8}A_1^{-\frac{5}{2}}\(\partial_tA_1\)^3v_t - \frac{3}{4}A_1^{-\frac{3}{2}}\partial_tA_1\partial_{tt}A_1v_t - \frac{3}{4}A_1^{-\frac{3}{2}}\(\partial_tA_1\)^2v_{tt} + \\ + \frac{1}{2}A_1^{-\frac{1}{2}}\partial_{ttt}A_1v_t + \frac{3}{2}A_1^{-\frac{1}{2}}\partial_{tt}A_1v_{tt} + \frac{3}{2}A_1^{-\frac{1}{2}}\partial_tA_1v_{ttt} + A_1^\frac{1}{2}v_{tttt}\period\end{array}\right.\end{equation}
It follows from \eqref{Phi_tttt_calc_H4}, \eqref{A_1^-sigma_H2}, \eqref{dt_A_1_H4}, \eqref{v_t_0_H2}, and \eqref{A_1^sigma_H4} that
\begin{equation}\label{Phi_tttt_H4}\abs{\Phi_{tttt}\(0\)} \lesssim \abs{v_t\(0\)} + \abs{\partial_{tt}A_1\(0\)} + \abs{v_{tt}\(0\)} + \abs{\partial_{ttt}A_1\(0\)} + \abs{v_{ttt}\(0\)} + \abs{v_{tttt}\(0\)}\period\end{equation}
All of the terms on the RHS except the fourth and sixth we already have under control. For the fourth term, we first relate $v_{ttt}$ to $\Phi_{ttt}$ by differentiating \eqref{v_tt_calc_H3} with respect to $t$ to obtain
\begin{equation}\label{v_ttt_calc_H4}v_{ttt} = \frac{3}{4}A_1^{-\frac{5}{2}}\(\partial_tA_1\)^2\Phi_t -\frac{1}{2}A_1^{-\frac{3}{2}}\partial_{tt}A_1\Phi_t - A_1^{-\frac{3}{2}}\partial_tA_1\Phi_{tt} + A_1^{-\frac{1}{2}}\Phi_{ttt}\period\end{equation}
Combining \eqref{A_1^-sigma_H2}, \eqref{dt_A_1_H4}, \eqref{Phi_t_H4}, \eqref{dtt_A_1_H4}, \eqref{Phi_tt_H4}, \eqref{Phi_ttt_H4}, \eqref{v_t_H4}, \eqref{v_H4}, and \eqref{v_tt_H4} leads to
\begin{equation}\label{v_ttt_H4}\ob\begin{array}{l}\abs{v_{ttt}} \lesssim v_t^2\abs{\Phi_t} + \abs{v}\abs{v_{tt}}\abs{\Phi_t} + \abs{v}\abs{v_t}\abs{\Phi_{tt}} + \abs{\Phi_{ttt}} \lesssim_\epsilon \chi_0\(r^{-1 - \epsilon} + \abs{\Phi_{ttt}}\) + \chi_\infty\(r^{-\frac{9}{2}} + \abs{\Phi_{ttt}}\)\comma \\ v_{ttt} \in L^pL^2\comma \, p \in \[1, \infty\]\period\end{array}\right.\end{equation}
Now, three direct differentiations of $A_1$, \eqref{v_H4}, \eqref{v_ttt_H4}, \eqref{v_t_H4}, \eqref{v_tt_H4}, \eqref{v_0_H2}, \eqref{v_ttt_0_H3}, \eqref{v_t_0_H2}, and \eqref{v_tt_0_H2} show
\begin{equation}\label{dttt_A_1_H4}\ob\begin{array}{l}\abs{\partial_{ttt}A_1} \lesssim \abs{v}\abs{v_{ttt}} + \abs{v_t}\abs{v_{tt}} \lesssim_\epsilon \chi_0\(r^{-1 - \epsilon} + \abs{\Phi_{ttt}}\) + \chi_\infty\(r^{-3} + r^{-\frac{3}{2}}\abs{\Phi_{ttt}}\)\comma \, \epsilon > 0\comma \\ \partial_{ttt}A_1 \in L^pL^q\comma \, \(p, q\) \in \[1, \infty\] \times \[1, 2\]\comma \\ \abs{\partial_{ttt}A_1\(0\)} \lesssim \chi_0r^{-1} + \chi_\infty r^{-3}\comma \\ \partial_{ttt}A_1\(0\) \in L^q\comma \, q \in \[1, 4\)\period\end{array}\right.\end{equation}
For the sixth term on the RHS of \eqref{Phi_tttt_H4}, we twice differentiate \eqref{v_wave_equation} with respect to $t$ and solve for $v_{tttt}$ to obtain
\begin{equation}\label{v_tttt_calc_H4}\abs{v_{tttt}\(0\)} \le \abs{\laplacian v_{tt}\(0\)} + \abs{\partial_{tt}\[F\(v\)\]\(0\)} \le \abs{\laplacian^2v\(0\)} + \abs{\laplacian\[F\(v\)\]\(0\)} + \abs{\partial_{tt}\[F\(v\)\]\(0\)}\period\end{equation}
By our initial conditions,
\begin{equation}\label{laplacian^2_v_0_H4}\laplacian^2v\(0\) \in L^2\period\end{equation}
To estimate the $\laplacian\[F\(v\)\]\(0\)$ term, we note that our initial conditions imply
\begin{equation}\label{laplacian_v_r_0_H4}\abs{\laplacian v_r\(0\)} \lesssim \chi_0r^{-1} + \chi_\infty r^{-\frac{3}{2}}\comma \quad \laplacian v_r\(0\) \in L^q\comma \, q \in \[2, 4\)\period\end{equation}
Two direct differentiations of $A_1$, a Maclaurin analysis, \eqref{v_0_H2}, and \eqref{laplacian_v_0_H2} show
\begin{equation}\label{laplacian_A_1_H4}\abs{\laplacian A_1\(0\)} \lesssim \abs{v\(0\)}\abs{\laplacian v\(0\)} \lesssim_\epsilon \chi_0r^{-\epsilon} + \chi_\infty r^{-3}\comma \, \epsilon > 0\comma \quad \laplacian A_1\(0\) \in L^q\comma \, q \in \[1, \infty\)\period\end{equation}
Using \eqref{A_1^-sigma_H2}, \eqref{gradient_A_1_H4}, \eqref{laplacian_A_1_H4}, \Cref{F_j_analyticity}, \eqref{v_0_H2}, \eqref{gradient_v_0_H2}, \eqref{laplacian_v_0_H2}, \eqref{v_t_0_H2}, \eqref{v_r_0_H2}, \eqref{gradient_v_r_0_H3}, \eqref{laplacian_v_t_0_H3}, \eqref{laplacian_v_r_0_H4}, and \eqref{1_-_A_1^-sigma_H2}, we calculate eventually that
\begin{equation}\label{laplacian_FFFF_H4}\abs{\laplacian\[F\(v\)\]\(0\)} \lesssim \chi_0r^{-1} + \chi_\infty r^{-\frac{7}{2}}\comma \quad \laplacian\[F\(v\)\]\(0\) \in L^q\comma \, q \in \(1, 4\)\period\end{equation}
To estimate the $\partial_{tt}\[F\(v\)\]\(0\)$ term, we note that our initial conditions, \eqref{v_wave_equation}, and \eqref{gradient_FFFF_H3} together imply
\begin{equation}\label{v_rtt_0_H4}\abs{v_{rtt}\(0\)} \lesssim \chi_0r^{-1} + \chi_\infty r^{-\frac{3}{2}}\comma \quad \laplacian v_{rtt}\(0\) \in L^q\comma \, q \in \[2, 4\)\period\end{equation}
Using \eqref{A_1^-sigma_H2}, \eqref{dt_A_1_H4}, \eqref{dtt_A_1_H4}, \Cref{F_j_analyticity}, \eqref{v_0_H2}, \eqref{v_t_0_H2}, \eqref{v_tt_0_H2}, \eqref{v_r_0_H2}, \eqref{v_rt_0_H3}, \eqref{v_ttt_0_H3}, \eqref{v_rtt_0_H4}, and \eqref{1_-_A_1^-sigma_H2}, we calculate eventually that
\begin{equation}\label{FFFF_tt_H4}\abs{\partial_{tt}\[F\(v\)\]\(0\)} \lesssim \chi_0r^{-1} + \chi_\infty r^{-\frac{7}{2}}\comma \quad \partial_{tt}\[F\(v\)\]\(0\) \in L^q\comma \, q \in \(1, 4\)\period\end{equation}
Putting \eqref{v_tttt_calc_H4}, \eqref{laplacian^2_v_0_H4}, \eqref{laplacian_FFFF_H4}, and \eqref{FFFF_tt_H4} together, we obtain
\begin{equation}\label{v_tttt_0_H4}v_{tttt}\(0\) \in L^2\period\end{equation}
It now follows from \eqref{Phi_tttt_H4}, \eqref{v_t_0_H2}, \eqref{dtt_A_1_H4}, \eqref{v_tt_0_H2}, \eqref{dttt_A_1_H4}, \eqref{v_ttt_0_H3}, and \eqref{v_tttt_0_H4} that the second term on the RHS of \eqref{Phi_ttt_strichartz_estimate} can be estimated by
\begin{equation}\label{Phi_ttt_strichartz_estimate_second_term}\norm{\Phi_{tttt}\(0\)}_{L^2} \lesssim 1\period\end{equation}
\subsection{Completion of the proof that $\Phi \in Y_4$}\label{subsec:Y_4_completion}
Combining \eqref{Phi_ttt_strichartz_estimate}, \eqref{Phi_ttt_strichartz_estimate_first_term}, \eqref{Phi_ttt_strichartz_estimate_second_term}, \eqref{Phi_ttt_strichartz_estimate_third_term}, \eqref{Phi_ttt_strichartz_estimate_fourth_term}, \eqref{Phi_ttt_strichartz_estimate_fifth_term}, and \eqref{Phi_ttt_strichartz_estimate_sixth_term}, we arrive at
\begin{equation}\label{Phi_ttt_strichartz_estimate_result}\norm{\Phi_{ttt}}_{L^pL^\frac{4p}{p - 1}} + \norm{\gradient\Phi_{ttt}}_{Y_0} + \norm{\Phi_{tttt}}_{Y_0} \lesssim 1\comma \quad p \in \[2, \infty\]\period\end{equation}
This proves $\gradient\Phi_{ttt}, \Phi_{tttt} \in Y_0$, leaving us with only $\laplacian^2\Phi$, $\gradient\laplacian\Phi_t$, and $\laplacian\Phi_{tt}$ to analyze. Looking at \eqref{Phi_tt_wave_equation}, we obtain
\begin{equation}\label{laplacian_Phi_tt_H4_calc}\abs{\laplacian\Phi_{tt}} \lesssim \abs{\Phi_{tttt}} + \abs{A_1^{-3}\partial_tA_1\Phi_t} + \abs{\(1 - A_1^{-2}\)\Phi_{tt}}\period\end{equation}
By \eqref{Phi_ttt_strichartz_estimate_result}, \eqref{A_1^-sigma_H2}, \eqref{dt_A_1_H4}, \eqref{Phi_t_H4}, \eqref{1_-_A_1^-sigma_H2}, and \eqref{Phi_tt_H4} this proves
\begin{equation}\label{laplacian_Phi_tt_H4}\norm{\laplacian\Phi_{tt}}_{Y_0} \lesssim 1\period\end{equation}
This leaves us with only $\laplacian^2\Phi$ and $\gradient\laplacian\Phi_t$ to analyze. Now, looking at \eqref{Phi_t_wave_equation}, taking the gradient, and rearranging, we obtain
\begin{equation}\label{gradient_laplacian_Phi_t_calc_H4}\gradient\laplacian\Phi_t = \gradient\Phi_{ttt} + 2A_1^{-3}\Phi_t\gradient A_1 - \(1 - A_1^{-2}\)\gradient\Phi_t\period\end{equation}
This clearly implies
\begin{equation}\label{gradient_laplacian_Phi_t_calc_2_H4}\abs{\gradient\laplacian\Phi_t} \lesssim \abs{\gradient\Phi_{ttt}} + \abs{A_1^{-3}\Phi_t\gradient A_1} + \abs{\(1 - A_1^{-2}\)\gradient\Phi_t}\period\end{equation}
We next record an estimate for $\gradient\Phi_t$, using the fact now that $\gradient\Phi_t \in Y_1$:
\begin{equation}\label{gradient_Phi_t_H4}\abs{\gradient\Phi_t} \lesssim \chi_0r^{-1} + \chi_\infty r^{-\frac{3}{2}}\comma \quad \gradient\Phi_t \in L^pL^q\comma \, \(p, q\) \in \[1, \infty\] \times \[2, 4\)\period\end{equation}
By \eqref{gradient_laplacian_Phi_t_calc_H4}, \eqref{Phi_ttt_strichartz_estimate_result}, \eqref{A_1^-sigma_H2}, \eqref{Phi_t_H4}, \eqref{gradient_A_1_H4}, \eqref{1_-_A_1^-sigma_H2}, and \eqref{gradient_Phi_t_H4}, this proves
\begin{equation}\label{gradient_laplacian_Phi_t_H4}\norm{\gradient\laplacian\Phi_t}_{Y_0} \lesssim 1\period\end{equation}
This leaves us with only $\laplacian^2\Phi$ to analyze. Going back to the wave equation \eqref{Phi_wave_equation} for $\Phi$, taking its Laplacian, and solving for $\laplacian^2\Phi$, we obtain
\begin{equation}\label{laplacian^2_Phi_calc_H4}\abs{\laplacian^2\Phi} \lesssim \abs{\laplacian\Phi_{tt}} + \abs{\laplacian\Phi} + \laplacian\(\int_0^vA_5^{-\frac{3}{2}} \, dy\) + \laplacian\(\phi_{\ge \frac{1}{2}}r^{-3}\)\period\end{equation}
We next record an estimate for $\laplacian\Phi$, using the fact now that $\laplacian\Phi \in Y_1$:
\begin{equation}\label{laplacian_Phi_H4}\abs{\laplacian\Phi} \lesssim \chi_0r^{-1} + \chi_\infty r^{-\frac{3}{2}}\comma \quad \laplacian\Phi \in L^pL^q\comma \, \(p, q\) \in \[1, \infty\] \times \[2, 4\)\period\end{equation}
Making use of \eqref{laplacian_Phi_tt_H4_calc}, \eqref{laplacian_Phi_H4}, and \Cref{phi_ge_1_2_properties}, we conclude
\begin{equation}\label{laplacian^2_Phi_calc_2_H4}\norm{\laplacian^2\Phi}_{Y_0} \lesssim 1 + \norm{\laplacian\(\int_0^vA_5^{-\frac{3}{2}} \, dy\)}_{Y_0}\period\end{equation}
Now we compute directly that
\begin{equation}\label{laplacian_int_A_5_calc_H4}\laplacian\(\int_0^vA_5^{-\frac{3}{2}} \, dy\) = -3A_1^{-\frac{5}{2}}\gradient A_1 \cdot \gradient v + A_1^{-\frac{3}{2}}\laplacian v + \frac{15}{4}\int_0^vA_5^{-\frac{7}{2}}\abs{\gradient A_5}^2 \, dy - \frac{3}{2}\int_0^vA_5^{-\frac{5}{2}}\laplacian A_5 \, dy\period\end{equation}
By \eqref{gradient_v_H4},
\begin{equation}\label{v_r_H4}\abs{v_r} \lesssim \abs{\gradient\Phi} + r\abs{v}^5 + \chi_\infty r^{-2}\abs{v} \lesssim_\epsilon \chi_0r^{-\epsilon} + \chi_\infty r^{-\frac{3}{2}}\comma \, \epsilon > 0\comma \quad v_r \in L^pL^q\comma \, \(p, q\) \in \[1, \infty\] \times \[2, \infty\)\period\end{equation}
Now, using \eqref{A_1^-sigma_H2}, \Cref{F_j_analyticity}, \eqref{v_H4}, \eqref{v_t_H4}, \eqref{v_r_H4}, and \eqref{1_-_A_1^-sigma_H2}, we obtain
\begin{equation}\label{FFFF_H4}\abs{F\(v\)} \lesssim_\epsilon \chi_0r^{-\epsilon} + \chi_\infty r^{-\frac{7}{2}}\comma \, \epsilon > 0\comma \quad F\(v\) \in L^pL^q\comma \, \(p, q\) \in \[1, \infty\] \times \(1, \infty\)\period\end{equation}
Using \eqref{v_wave_equation}, \eqref{v_tt_H4}, and \eqref{FFFF_H4}, we obtain
\begin{equation}\label{laplacian_v_H4}\abs{\laplacian v} \lesssim \abs{v_{tt}} + \abs{F\(v\)} \lesssim \chi_0r^{-1} + \chi_\infty r^{-\frac{3}{2}}\comma \quad \laplacian v \in L^pL^q\comma \, p \in \[1, \infty\)\comma q \in \[2, \min\ob8, \frac{4p}{p - 1}\cb\]\period\end{equation}
Combining \eqref{laplacian_int_A_5_calc_H4}, \eqref{A_1^-sigma_H2}, \eqref{gradient_A_1_H4}, \eqref{gradient_v_H4}, and \eqref{laplacian_v_H4}, we have
\begin{equation}\label{laplacian_int_A_5_calc_2_H4}\norm{\laplacian\(\int_0^vA_5^{-\frac{3}{2}} \, dy\)}_{Y_0} \lesssim 1 + \norm{\int_0^v\abs{\gradient A_5}^2 \, dy}_{Y_0} + \norm{\int_0^v\laplacian A_5 \, dy}_{Y_0}\period\end{equation}
Maclaurin analysis shows $\abs{\gradient A_5}^2 \lesssim r^2y^8$ and $\abs{\laplacian A_5} \lesssim y^4$. It follows that $\int_0^v\abs{\gradient A_5}^2 \, dy \lesssim r^2\abs{v}^9 \lesssim r^2\jbracket{r}^{-\frac{27}{2}} \lesssim \jbracket{r}^{-\frac{23}{2}}$ and $\abs{\int_0^v\laplacian A_5 \, dy} \lesssim \abs{v}^5 \lesssim \jbracket{r}^{-\frac{15}{2}}$. Clearly, then, we can upgrade \eqref{laplacian_int_A_5_calc_2_H4} to
\begin{equation}\label{laplacian_int_A_5_calc_3_H4}\norm{\laplacian\(\int_0^vA_5^{-\frac{3}{2}} \, dy\)}_{Y_0} \lesssim 1\period\end{equation}
Now combining \eqref{laplacian^2_Phi_calc_2_H4} and \eqref{laplacian_int_A_5_calc_3_H4} gives
\begin{equation}\label{laplacian^2_Phi_H4}\norm{\laplacian^2\Phi}_{Y_0} \lesssim 1\period\end{equation}
By \eqref{Phi_ttt_strichartz_estimate_result}, \eqref{laplacian_Phi_tt_H4}, \eqref{gradient_laplacian_Phi_t_H4}, and \eqref{laplacian^2_Phi_H4}, then, we have proven $\Phi \in Y_4$.
\section{Final steps}\label{sec:final_steps}
The fact that $\Phi \in Y_4$ implies \eqref{Phi_t_H4} and \eqref{gradient_Phi_H4} can be upgraded to
\begin{alignat}{3}
\label{Phi_t_final}& \abs{\Phi_t} \lesssim \jbracket{r}^{-\frac{3}{2}}\comma & \quad & \Phi_t \in L^pL^q\comma \, \(p, q\) \in \[1, \infty\] \times \[2, \infty\]\comma \\
\label{gradient_Phi_final}& \abs{\gradient\Phi} \lesssim \jbracket{r}^{-\frac{3}{2}}\comma & \quad & \gradient\Phi \in L^pL^q\comma \, \(p, q\) \in \[1, \infty\] \times \[2, \infty\]\comma
\end{alignat}
respectively. In consequence of this, \eqref{v_t_H4} and \eqref{gradient_v_H4} can be upgraded to
\begin{alignat}{3}
\label{v_t_final}& \abs{v_t} \lesssim \abs{\Phi_t} \lesssim \jbracket{r}^{-\frac{3}{2}}\comma & \quad & v_t \in L^pL^q\comma \, \(p, q\) \in \[1, \infty\] \times \[2, \infty\]\comma \\
\label{gradient_v_final}& \abs{\gradient v} \lesssim \abs{\gradient\Phi} + r\abs{v}^5 \lesssim \jbracket{r}^{-\frac{3}{2}}\comma & \quad & \gradient v \in L^pL^q\comma \, \(p, q\) \in \[1, \infty\] \times \[2, \infty\]\comma
\end{alignat}
respectively. Now, \eqref{v_H4}, \eqref{v_t_final}, and \eqref{gradient_v_final} together imply $\(\jbracket{r}v, \jbracket{r}v_t, \jbracket{r}\gradient v\) \in L^\infty\(S_{I, 4}\)$. By \Cref{v_continuation_criterion_theorem}, this implies $T < T^*$. But since $T < \infty$ was chosen arbitrarily, we may now conclude $T^* = \infty$. Using the reasoning laid out in \Cref{u_gwp_remark}, this proves \Cref{u_global_existence_theorem}.
\section*{Acknowledgements}\label{sec:ack}
The author was supported in part by the National Science Foundation Career grant DMS-0747656.
\bibliographystyle{amsalpha}
\bibliography{references}
\end{document}